\documentclass[10pt,leqno]{amsart}

\usepackage{amsmath}
\usepackage{amsfonts}
\usepackage{amssymb}
\usepackage{graphicx}
\usepackage{color}
\usepackage{hyperref}
\usepackage{xcolor}

\newtheorem{theorem}{Theorem}[section]

\newtheorem*{theorem*}{Theorem}
\newtheorem{lemma}{Lemma}[section]

\newtheorem{proposition}{Proposition}[section]

\newtheorem{example}[theorem]{Example}

\newcommand{\R}{\mathbb{R}}

\def\Ric{\text{Ric}}

\def\a{\alpha}
\def\l{\lambda}

\def\C{\mathbb{C}}
\def\R{\mathbb{R}}
\def\Z{\mathbb{Z}}

\def\S{\mathbb{S}}
\def\H{\mathbb{H}}
\def\CP{\mathbb{CP}}
\def\CH{\mathbb{CH}}

\def\vp{\varphi}

\def\k{\kappa}

\def\H{\mathbb{H}}
\def\id{\operatorname{id}}
\def\Ric{\operatorname{Ric}}

\def\tr{\operatorname{tr}}

\def\spn{\operatorname{span}}
\newcommand{\rvline}{\hspace*{-\arraycolsep}\vline\hspace*{-\arraycolsep}}

\newcommand{\SO}{{\mathsf{SO}}}
\newcommand{\U}{{\mathsf{U}}}
\numberwithin{equation}{section}
\makeatletter
\newcommand*\owedge{\mathpalette\@owedge\relax}
\newcommand*\@owedge[1]{%
  \mathbin{%
    \ooalign{%
      $#1\m@th\bigcirc$\cr
      \hidewidth$#1\m@th\wedge$\hidewidth\cr
    }%
  }%
}
\makeatother

\begin{document}

\title[Curvature operator of the second kind]{Product manifolds and the curvature operator of the second kind}

\author{Xiaolong Li}\thanks{The author's research is partially supported by Simons Collaboration Grant \#962228
and a start-up grant at Wichita State University}
\address{Department of Mathematics, Statistics and Physics, Wichita State University, Wichita, KS, 67260}
\email{xiaolong.li@wichita.edu}



\begin{abstract}
We investigate the curvature operator of the second kind on product Riemannian manifolds and obtain some optimal rigidity results. For instance, we prove that the universal cover of an $n$-dimensional non-flat complete locally reducible Riemannian manifold with $(n+\frac{n-2}{n})$-nonnegative (respectively, $(n+\frac{n-2}{n})$-nonpositive) curvature operator of the second kind must be isometric to $\S^{n-1}\times \R$ (respectively, $\H^{n-1}\times \R$) up to scaling. We also prove analogous optimal rigidity results for $\S^{n_1}\times \S^{n_2}$ and $\H^{n_1}\times \H^{n_2}$, $n_1,n_2 \geq 2$, among product Riemannian manifolds, as well as for $\CP^{m_1}\times \CP^{m_2}$ and $\CH^{m_1}\times \CH^{m_2}$, $m_1,m_2\geq 1$, among product K\"ahler manifolds. Our approach is pointwise and algebraic. 
\end{abstract}

\maketitle

\section{Introduction}

On a Riemannian manifold $(M^n, g)$, the \textit{curvature operator of the second kind} at $p \in M$ refers to the symmetric bilinear form $\mathring{R}: S^2_0 (T_pM)\times S^2_0 (T_pM) \to  \R$
defined by $$\mathring{R}(\vp, \psi) = R_{ijkl}\vp_{il}\psi_{jk},$$ 
where $S^2_0(T_pM)$ is the space of traceless symmetric two-tensors on $T_pM$.
The terminology is due to Nishikawa \cite{Nishikawa86}. Early works studying this notion of curvature operator are \cite{BK78}, \cite{OT79}, \cite{Nishikawa86} and \cite{Kashiwada93}.  

In the past year, the notion of the curvature operator of the second kind has received much attention. See the recent works \cite{CGT21}, \cite{Li21, Li22PAMS, Li22JGA, Li22Kahler}, \cite{NPW22} and \cite{NPWW22}.  
In particular, the longstanding conjecture of Nishikawa \cite{Nishikawa86}, which asserts that a closed Riemannian manifold with positive curvature operator of the second kind is diffeomorphic to a spherical space form and a closed Riemannian manifold with nonnegative curvature operator of the second kind is diffeomorphic to a Riemannian locally symmetric space, has been resolved in \cite{CGT21}, \cite{Li21} and \cite{NPW22}, under weaker assumptions but with stronger conclusions. More precisely, it is known now that 
\begin{theorem}\label{thm 3 positive}
Let $(M^n,g)$ be a closed Riemannian manifold of dimension $n\geq 3$. 
\begin{enumerate}
    \item If $(M^n,g)$ has three-positive curvature operator of the second kind, then $M$ is diffeomorphic to a spherical space form.
    \item If $(M^n,g)$ has three-nonnegative curvature operator of the second kind, then $M$ is either flat or diffeomorphic to a spherical space form. 
\end{enumerate}
\end{theorem}

The key observation in \cite{CGT21} is that two-positive curvature operator of the second kind implies strictly PIC1 (i.e. $M\times \R$ has positive isotropic curvature). 
This is sufficient to solve the positive case of Nishikawa's conjecture, as one can appeal to a result of Brendle \cite{Brendle08} stating that the Ricci flow on a compact manifold starting with a strictly PIC1 metric exists for all time and converge to a limit metric with constant positive sectional curvature. 
Shortly after, the author showed that strictly PIC1 is implied by three-positivity of the curvature operator of the second kind, thus getting an improvement of the result in \cite{CGT21}. 
To deal with the nonnegative case, the author \cite{Li21} reduces the problem to the locally irreducible case by proving that a complete $n$-dimensional Riemannian manifold with $n$-nonnegative curvature operator of the second kind is either flat or locally irreducible (see also Theorem \ref{thm split} for an improvement of this result). Finally, K\"ahler manifolds are ruled out using \cite[Theorem 1.9]{Li21} (see also \cite{Li22Kahler} for an improvement) and irreducible symmetric spaces are ruled out using \cite[Theorem A]{NPW22}. 
We refer the reader to \cite{Li22JGA} or \cite{Li22Kahler} for a detailed account on the notion of the curvature operator of the second kind, as well as some recent developments. 

The aim of this article is to study the curvature operator of the second kind on product manifolds and obtain some optimal rigidity results. We first recall the following definition. 
For $\a \in [1,\frac{(n-1)(n+2)}{2}]$, we say $(M^n,g)$ has $\alpha$-positive (respectively, $\a$-nonnegative) curvature operator of the second kind for if for any $p\in M$ and any orthonormal basis $\{\vp_i\}_{1 \leq i \leq \frac{(n-1)(n+2)}{2}}$ of $S^2_0(T_pM)$, it holds that
\begin{equation}\label{alpha positive def}
     \sum_{i=1}^{\lfloor \a \rfloor} \mathring{R}(\vp_i,\vp_i) +(\a -\lfloor \a \rfloor) \mathring{R}(\vp_{\lfloor \a \rfloor+1},\vp_{\lfloor \a \rfloor+1}) > (\text{respectively,} \geq) \  0. 
\end{equation} 
Here and in the rest of this article, $\lfloor x \rfloor$ denotes the floor function defined by $$ \lfloor x \rfloor:= \max \{ m \in \Z : m \leq x \}.$$
When $\a=k$ is an integer, this reduces to the usual definition, which means the sum of the smallest $k$ eigenvalues of the matrix $\mathring{R}(\vp_i,\vp_j)$ is positive/nonnegative for any orthonormal basis $\{\vp_i\}_{i=1}^N$ of $S^2_0(T_pM)$. 
Similarly, $(M^n,g)$ is said to have $\a$-negative (respectively, $\a$-nonpositive) if the direction of the inequality \eqref{alpha positive def} is reversed.

The first main result of this article is the following rigidity result for $\mathbb{S}^{n-1}\times \R$ and $\mathbb{H}^{n-1}\times \R$. Here and in the rest of this article, $\S^n$ and $\H^n$, $n\geq 2$, denote the $k$-dimensional sphere and hyperbolic space with constant sectional curvature $1$ and $-1$, respectively.
\begin{theorem}\label{thm SnXR}
Let $(M^n,g)$ be a nonflat complete locally reducible Riemannian manifold of dimension $n\geq 4$. 
\begin{enumerate}
    \item If $M$ has $(n+\frac{n-2}{n})$-nonnegative curvature operator of the second kind, then the universal cover of $M$ is, up to scaling, isometric to the $\mathbb{S}^{n-1}\times \R$. 
    \item If $M$ has $(n+\frac{n-2}{n})$-nonpositive curvature operator of the second kind, then the universal cover of $M$ is, up to scaling, isometric to the $\mathbb{H}^{n-1}\times \R$. 
\end{enumerate}
\end{theorem}

Closely related is the following holonomy restriction theorem in the spirit of \cite{NPWW22}. 
\begin{theorem}\label{thm holonomy}
Let $(M^n,g)$ be a (not necessarily complete) Riemannian manifold of dimension $n\geq 3$. Suppose that $(M,g)$ has $\alpha$-nonnegative or $\alpha$-nonpositive curvature operator of the second kind for some $\alpha  < n+\frac{n-2}{n}$. Then either $M$ is flat or the restricted holonomy of $M$ is $\SO(n)$.
\end{theorem}

Theorems \ref{thm SnXR} and \ref{thm holonomy} improve previous results obtained in \cite{Li21} and \cite{NPWW22}. The author \cite[Theorem 1.8]{Li21} proved that an $n$-dimensional complete Riemannian manifold with $n$-nonnegative curvature operator of the second kind is either flat or locally reducible. 
This result plays a significant role in resolving the nonnegative part of Nishikawa's conjecture in \cite{Li21}, as it allows one to reduce the problem to the locally irreducible setting. 
Moreover, a slight modification of the proof yields the same conclusion under $n$-nonpositive curvature operator of the second kind. 
The method used in \cite{Li21} is pointwise and algebraic. 
In \cite{NPWW22}, it is shown that if the curvature operator of the second kind of an $n$-dimensional Riemannian manifold, not necessarily complete, is $n$-nonnegative or $n$-nonpositive, then either the restricted holonomy of $M$ is $\SO(n)$ or $M$ is flat. 
This is a generalization of the author's result in \cite{Li21}. 
The approach of \cite{NPWW22} is local and the key idea is that unless the restricted holonomy is generic, there exists a parallel form, at least locally on the manifold, but on the other hand, Bochner technique with the curvature assumption implies that no such local parallel form exists unless the manifold is flat. 
Theorem \ref{thm holonomy} can also be viewed as supporting evidence of the author's conjecture in \cite{Li22JGA}: a closed $n$-dimensional Riemannian manifold with $\left(n+\frac{n-2}{n}\right)$-positive curvature operator of the second kind is diffeomorphic to a spherical space form. 
%The example $\S^{n-1}\times \S^1$ shows that the result cannot be improved to $(n+1)$-nonnegative curvature operator of the second kind, 
%Theorems \ref{thm SnXR} and \ref{thm holonomy} are further generalizations of the above-mentioned results. 

We would like to point out that the number $n+\frac{n-2}{n}$ in Theorems \ref{thm SnXR} and \ref{thm holonomy} is optimal in all dimensions, in view of the fact that $\S^{n-1}\times \R$ and $\H^{n-1}\times \R$ have $\left(n+\frac{n-2}{n}\right)$-nonnegative and $\left(n+\frac{n-2}{n}\right)$-nonpositive curvature operator of the second kind, respectively, and they both have restricted holonomy $\SO(n-1)$. In dimension four, $\CP^2$ and $\CH^2$ have $4\frac{1}{2}$-nonnegative and $4\frac{1}{2}$-nonpositive curvature operator of the second kind, respectively, and they both have restricted holonomy $\mathsf{U}(2)$. 

As a generalization of Theorem \ref{thm 3 positive}, the author proved in \cite{Li22JGA} that a closed Riemannian manifold of dimension $n\geq 4$ with $4\frac{1}{2}$-positive curvature operator of the second kind is homeomorphic to a spherical space form. This is obtained by showing that $4\frac{1}{2}$-positive curvature operator of the second kind implies positive isotropic curvature and $\left(n+\frac{n-2}{n}\right)$-positive curvature operator of the second kind implies positive Ricci curvature, and then making use of the work of Micallef and Moore \cite{MM88}. 
A classification result of closed manifolds with $4\frac{1}{2}$-nonnegative curvature operator of the second kind was also obtained in \cite[Theorem 1.4]{Li22JGA}. Using Theorem \ref{thm SnXR}, together with \cite[Theorem 1.2]{Li22Kahler} and \cite[Theorem B]{NPW22}, we get an improvement of \cite[Theorem 1.4]{Li22JGA}. 
\begin{theorem}\label{thm 4.5 nonnegative}
Let $(M^n,g)$ be a closed non-flat Riemannian manifold of dimension $n\geq 4$. 
Suppose that $M$ has $4\frac 1 2$-nonnegative curvature operator of the second kind, then one of the following statements holds:
\begin{enumerate}
    \item $M$ is homeomorphic (diffeomorphic if either $n=4$ or $n\geq 12$) to a spherical space form;
    \item $n=4$ and $M$ is isometric to $\mathbb{CP}^2$ with Fubini-Study metric up to scaling;
    \item $n=4$ and the universal cover of $M$ is isometric to $\mathbb{S}^3 \times \R$ up to scaling.
\end{enumerate}
\end{theorem}

%\begin{theorem}\label{thm SnXR}
%Let $(M^n,g)$ be a Riemannian manifold of dimension $n\geq 4$. Suppose $M$ has $\a$-nonnegative or $\a$-nonpositive curvature operator of the second kind for $\a <n+\frac{n-2}{n}$, respectively, then either the restricted holonomy of $M$ is $\SO(n)$ or $M$ is flat. 
%\begin{enumerate}
%    \item If $M$ has $(n+\frac{n-2}{n})$-nonnegative curvature operator of the second kind, then the universal cover of $M$ is, up to scaling, isometric to the $\mathbb{S}^{n-1}\times \R$. 
%    \item If $M$ has $(n+\frac{n-2}{n})$-nonpositive curvature operator of the second kind, then the universal cover of $M$ is, up to scaling, isometric to the $\mathbb{H}^{n-1}\times \R$. 
%\end{enumerate}
%\end{theorem}

Our second main result is about the rigidity of $\S^{n_1}\times \S^{n_2}$ and $\H^{n_1}\times \H^{n_2}$ among product Riemannian manifolds.
\begin{theorem}\label{thm product spheres}
    Let $(M^{n_i}_i,g_i)$ be a Riemannian manifold of dimension $n_i \geq 2$ for $i=1,2$, and let $(M^{n_1+n_2}, g)=(M^{n_1}_1 \times M^{n_2}_2, g_1 \oplus g_2)$.
    Set 
    \begin{equation}\label{eq A_n_1n_2 def}
    A_{n_1,n_2}:=1+n_1n_2+\frac{n_1(n_2-1)+n_2(n_1-1)}{n_1+n_2}.
\end{equation}
Then 
    \begin{enumerate}
        \item If $M$ has $\a$-nonnegative or $\a$-nonpositive curvature operator of the second kind for some $\a< A_{n_1,n_2}$, then $M$ is flat. 
        \item If $M$ has $A_{n_1,n_2}$-nonnegative curvature operator of the second kind, then 
        both $M_1$ and $M_2$ have constant sectional curvature $c\geq 0$. 
        %either $M$ is flat, or the universal cover of $M$ is isometric to $\S^{n_1}\times \S^{n_2}$ up to scaling. 
        \item If $M$ has $A_{n_1,n_2}$-nonpositive curvature operator of the second kind, then
        both $M_1$ and $M_2$ have constant sectional curvature $c\leq 0$. 
        %then either $M$ is flat, or the universal cover of $M$ is isometric to $\H^{n_1}\times \H^{n_2}$ up to scaling.
    \end{enumerate}
If $M$ is further assumed to be complete and nonflat, then the universal cover of $M$ is isometric to $\S^{n_1}\times \S^{n_2}$ in part (2) and $\H^{n_1}\times \H^{n_2}$ in part (3), up to scaling. 
\end{theorem}
%\begin{remark}
%    Theorem \ref{thm product spheres} remains valid for $n_1=1$ or $n_2=1$, as long as $\S^1$ and $\H^1$ are replaced by $\R$. 
%\end{remark}

The author proved in \cite[Proposition 5.1]{Li21} that an $n$-manifold with $(k(n-k)+1)$-nonnegative curvature operator of the second kind cannot split off a $k$-dimensional factor with $1\leq k \leq n/2$, unless it is flat. The number $k(n-k)+1$ is only optimal for some special $n$ and $k$. Combining Theorem \ref{thm SnXR} and Theorem \ref{thm product spheres}, we get the following generalization, which is optimal for any $n$ and $1\leq k \leq n/2$. 
\begin{theorem}\label{thm split}
An $n$-dimensional Riemannian manifold with $\a$-nonnegative or $\a$-nonpositive curvature operator of the second kind for some 
\begin{equation*}
    \a < k(n-k)+\frac{2k(n-k)}{n}
\end{equation*}
cannot locally split off a $k$-dimensional factor with $1\leq k \leq n/2$, unless it is flat. 
\end{theorem}

%Theorems \ref{thm SnXR} and \ref{thm product spheres} provide an extension of this result optimal for any $n_1, n_2 \geq 2$, as $\S^{n_1}\times \S^{n_2}$ (respectively, $\H^{n_1}\times \H^{n_2}$)  has $A_{n_1,n_2}$-nonnegative (respectively, $A_{n_1,n_2}$-nonpositive) curvature operator of the second kind. 

In another direction, the curvature operator of the second kind has been investigated for K\"ahler manifolds in \cite{BK78}, \cite{Li21, Li22PAMS, Li22Kahler} and \cite{NPWW22}. 
For instance, it was shown in \cite{Li22Kahler} that $m$-dimensional K\"ahler manifolds with $\frac{3}{2}(m^2-1)$-nonnegative curvature operator of the second kind have constant nonnegative holomorphic sectional curvature, and a closed $m$-dimensional K\"ahler manifold with $\left(\frac{3m^3-m+2}{2m}\right)$-positive curvature operator of the second kind has positive orthogonal bisectional curvature, thus being biholomorphic to $\CP^m$. 
Here we prove the following rigidity result for $\CP^{m_1}\times \CP^{m_2}$ and $\CH^{m_1}\times \CH^{m_2}$ (all equipped with their standard metrics). 
\begin{theorem}\label{thm product CPm}
    Let $(M^{m_i}_i,g_i)$ be a K\"ahler manifold of complex dimension $m_i \geq 1$ for $i=1,2$, and let  $(M^{m_1+m_2}, g)=(M^{m_1}_1 \times M^{m_2}_2, g_1 \oplus g_2)$.
    Set 
    \begin{equation}\label{eq B_m_1m_2 def}
    B_{m_1,m_2}:=4m_1m_2+\frac{3}{2}\left(m_1^2+m_2^2\right)+\frac{m_1m_2}{m_1+m_2}.
\end{equation}
Then 
    \begin{enumerate}
        \item If $M$ has $\a$-nonnegative or $\a$-nonpositive curvature operator of the second kind for some $\a < B_{m_1,m_2}$, then $M$ is flat. 
        \item If $M$ has $ B_{m_1,m_2}$-nonnegative curvature operator of the second kind, then 
        both $M_1$ and $M_2$ have constant holomorphic sectional curvature $c\geq 0$. 
        %either $M$ is flat, or the universal cover of $M$ is isometric to $\CP^{m_1}\times \CP^{m_2}$ up to scaling. 
        \item If $M$ has $ B_{m_1,m_2}$-nonpositive curvature operator of the second kind, then 
        both $M_1$ and $M_2$ have constant holomorphic sectional curvature $c\leq 0$. 
        %then either $M$ is flat, or the universal cover of $M$ is isometric to $\CH^{m_1}\times \CH^{m_2}$ up to scaling. 
    \end{enumerate}
If $M$ is further assumed to be complete and nonflat, then the universal cover of $M$ is isometric to $\CP^{m_1}\times \CP^{m_2}$ in part (2) and $\CH^{m_1}\times \CH^{m_2}$ in part (3), up to scaling.  
\end{theorem}

Our investigation of the curvature operator of the second kind on product manifolds is motivated not only by the above-mentioned optimal rigidity results, but also by the fact that the spectrum of $\mathring{R}$ are known only for a few examples: space forms with constant sectional curvature, K\"ahler and quaternion-K\"ahler space forms (\cite{BK78}), $\S^2 \times \S^2$ (\cite{CGT21}), $\S^{n-1}\times \R$ (\cite{Li21}), $\S^p\times \S^q$ (\cite{NPWW22}). 
In the present paper, we determine the spectrum of $\mathring{R}$ for a class of product manifold by proving the following theorem. 
\begin{theorem}\label{thm product eigenvalue}
Let $(M_i,g_i)$ be an $n_i$-dimensional Einstein manifold with $\Ric(g_i)=\rho_i g_i$ and $n_i\geq 1$ for $i=1,2$. 
Denote by $\mathring{R}_{i}$ the curvature operator of the second kind of $M_i$ for $i=1,2$, and $\mathring{R}$ the curvature operator of the second kind of the product manifold 
$$(M^{n_1+n_2}, g)=(M_1^{n_1} \times M_2^{n_2}, g_1 \oplus g_2).$$ 
Then the eigenvalues of $\mathring{R}$ are precisely those of $\mathring{R}_{1}$ and $\mathring{R}_{2}$, and $0$ with multiplicity $n_1 n_2$, and $-\frac{n_1\rho_2+n_2\rho_1}{n_1+n_2} $ with multiplicity one.
\end{theorem}

Theorem \ref{thm product eigenvalue} enables us to determine the spectrum of the curvature operator of the second kind on $(M_1, g_1) \times (M_2, g_2)$, with $(M_i, g_i)$ being either a space form with constant sectional curvature or a K\"ahler space form with constant holomorphic sectional curvature for $i=1,2$. Examples are listed at the end of Section 2. 
More generally, Theorem \ref{thm product eigenvalue} can be applied repeatedly to calculate the spectrum of $\mathring{R}$ for product manifolds of the form $(M_1,g_1) \times  \cdots \times (M_k,g_k)$, provided that each $(M_i,g_i)$ is Einstein and the eigenvalues of $\mathring{R}_i$ on $M_i$ are known. 

Let's discuss the strategy of our proofs. 
The key idea to prove Theorems \ref{thm SnXR}, \ref{thm product spheres} and \ref{thm product CPm} is to use the corresponding borderline example, such as $\S^{n-1}\times \R$, $\S^{n_1}\times \S^{n_2}$ or $\CP^{m_1}\times \CP^{m_2}$, as a model space and apply $\mathring{R}$ to the eigenvectors of the curvature operator of the second kind on the model space.  
This idea has been successfully employed by the author in \cite{Li22JGA} with $\CP^2$ and $\S^3 \times \R$ as model spaces, in \cite{Li22PAMS} with $\S^2 \times \S^2$ as the model space and in \cite{Li22Kahler} with $\CP^m$ and $\CP^{m-1}\times \CP^1$ as model spaces. 
With the right choice of model space, this strategy leads to optimal results as the inequalities are all achieved as equalities on the model space. 
Theorem \ref{thm split} is essentially a consequence of Theorems \ref{thm SnXR} and \ref{thm product spheres}. 
The proof of Theorem \ref{thm holonomy} uses Berger's classification of restricted holonomy groups, together with Propositions \ref{prop 4.3} and \ref{prop 5.3}, and results in \cite{Li22Kahler} and \cite{NPWW22}. 
The proof of Theorem \ref{thm product eigenvalue} replies on the fact that when both factor are Einstein, we can choose an orthonormal basis of the space of traceless symmetric two-tensors that diagonalizes the curvature operator of the second kind on the product manifold. 

At last, we emphasize that our approach is of pointwise nature, and therefore, many of our results are of pointwise nature and the completeness of the metric is not required. Another feature is that our proofs are purely algebraic and work equally well for nonpositivity conditions on $\mathring{R}$.

The article is organized as follows. 
In Section 2, we study the curvature operator of the second kind on product Riemannian manifolds and prove Theorem \ref{thm product eigenvalue}. 
We present the proofs of Theorems \ref{thm SnXR} and \ref{thm 4.5 nonnegative} in Section 3. 
The proofs of Theorems \ref{thm product spheres} and \ref{thm split} are given in Section 4.  
In Section 5, we prove Theorem \ref{thm holonomy}. 
Section 6 is devoted to the proof of Theorem \ref{thm product CPm}. 

Throughout this paper, all manifolds are assumed to be connected. We use the same notations and conventions as in \cite{Li21} or \cite{Li22JGA} or \cite{Li22Kahler}.

\section{Product Manifolds}
In this section, we study the curvature operator of the second kind on product Riemannian manifolds and prove Theorem \ref{thm product eigenvalue}. 

%Combining the above three lemmas together, we obtain that 
%theorem, which is an equivalent statement to Theorem \ref{thm product eigenvalue}. 
%\begin{theorem}\label{prop 3.0}
%Let $(M_i,g_i)$ be an $n_i$-dimensional Einstein manifold with $\Ric(g_i)=\rho_i g_i$ and $n_i\geq 1$ for $i=1,2$. 
%Denote by $\mathring{R}^{i}$ the curvature operator of the second kind of $M_i$ for $i=1,2$, and %$\mathring{R}$ the curvature operator of the second kind of the product manifold 
%$$(M^{n_1+n_2}, g)=(M_1^{n_1} \times M_2^{n_2}, g_1 \oplus g_2).$$ 
%Then the eigenvalues of $\mathring{R}$ are precisely those of $\mathring{R}^{1}$ and $\mathring{R}^{2}$, and $0$ with multiplicity $n_1 n_2$, and $-\frac{n_1\rho_2+n_2\rho_1}{n_1+n_2} $ with multiplicity one.
%\end{theorem}

Recall that for Riemannian manifolds $(M_1,g_1)$ and $(M_2,g_2)$, the product metric $g_1\oplus g_2$ on $M_1\times M_2$ is defined by
\begin{equation*}\label{eq product metric}
    g(X_{1}+X_{2},Y_{1}+Y_{2})=g_{1}(X_{1},Y_{1})+g_{2}(X_{2},Y_{2})
\end{equation*}
for $X_i,Y_i\in T_{p_i}M_i$
under the natural identification $$T_{(p_{1},p_{2})}(M_{1}\times M_{2})=T_{p_{1}}M_{1}\oplus T_{p_{2}}M_{2}.$$
Let $R$ denote the Riemann curvature tensor of $M=M_1\times M_2$, and $R_1,R_2$ denote the Riemann curvature tensor of $M_1$ and $M_2$, respectively. 
Then one can relate $R$, $R_1$ and $R_2$ by
\begin{eqnarray*}\label{eq product curvature}
    && R(X_1+X_2,Y_1+Y_2,Z_1+Z_2,W_1+W_2) \nonumber \\
    & =& R_1(X_1,Y_1,Z_1,W_1)+R_2(X_2,Y_2,Z_2,W_2),
\end{eqnarray*}
where $X_i,Y_i,Z_i,W_i \in TM_i$. 
As the reader will see, the above equation, which is a consequence of the product structure, plays a significant role in this section. 

From now on, let's focus on a single point in a product manifold and work in an purely algebraic way.
For $i=1,2$, let $(V_i,g_i)$ be an Euclidean vector space of dimension $n_i \geq 1$. 
The product space $V=V_1\times V_2$ will be naturally identified with $V_1 \oplus V_2$ via the isomorphism $(X_1,X_2) \to X_1 +X_2$ for $X_i \in V_i$. 
The product metric on $V$, denoted by $g=g_1 \oplus g_2$, is defined by 
\begin{equation}\label{eq g product}
    g(X_{1}+X_{2},Y_{1}+Y_{2})=g_{1}(X_{1},Y_{1})+g_{2}(X_{2},Y_{2})
\end{equation}
for $X_i,Y_i \in V_i$. 

Denote by $S^2_B(\Lambda^2 V)$ the space of algebraic curvature operators on $(V,g)$. That is to say, $R\in S^2_B(\Lambda^2 V)$ is a symmetric two-tensor on the space of two-forms $\Lambda^2 V$ on $V$ and $R$ also satisfies the first Bianchi identity. 
Given $R_i \in S^2_B(\Lambda^2 V_i)$ for $i=1,2$, we define $R\in S^2_B(\Lambda^2 V) $ by
\begin{eqnarray}\label{eq R product}
    && R(X_1+X_2,Y_1+Y_2,Z_1+Z_2,W_1+W_2) \\
    & =& R_1(X_1,Y_1,Z_1,W_1)+R_2(X_2,Y_2,Z_2,W_2), \nonumber 
\end{eqnarray}
for $X_i,Y_i,Z_i,W_i \in V_i$.
Throughout this paper, we simply write 
$$R=R_1 \oplus R_2$$
whenever $R$, $R_1$ and $R_2$ are related by \eqref{eq R product}.  
We also denote by $\mathring{R}$, $\mathring{R}_1$ and $\mathring{R}_2$ the associated curvature operator of the second kind for $R=R_1\oplus R_2$, $R_1$ and $R_2$, respectively. 

The key result of this section is the following proposition.
\begin{proposition}\label{prop 3.5}
Let $R_i \in S^2_B(\Lambda^2 V_i)$ for $i=1,2$ with $\dim(V_i)=n_i \geq 1$ and let $R=R_1 \oplus R_2$. 
If $\Ric(R_i)=\rho_i g_i$ for $i=1,2$, then the eigenvalues of $\mathring{R}$ are precisely those of $\mathring{R}_1$ and $\mathring{R}_2$, together with $0$ with multiplicity $n_1n_2$ and $-\frac{n_2 \rho_1 +n_1 \rho_2}{n_1+n_2}$ with multiplicity one.
\end{proposition}

We will present the proof of Proposition \ref{prop 3.5} after we establish the following three lemmas. In the rest of this section, $\mathring{R}$ acts on the space of symmetric two-tensors $S^2(V)$ via 
\begin{equation*}
    \mathring{R}(\vp)_{ij}=\sum_{k,l=1}^n R_{ijkl}\vp_{kl}.
\end{equation*}
Note that the curvature operator of the second kind (defined as a symmetric bilinear form in the Introduction) is equivalent to the symmetric bilinear form associated to the self-adjoint operator $\pi \circ \mathring{R} :S^2_0(V) \to S^2_0(V)$, where $\pi:S^2(V)\to S^2_0(V)$ is the projection map. This can be seen as 
\begin{equation*}
   \mathring{R}(\vp,\psi)= \langle \mathring{R}(\vp),\psi \rangle =\langle(\pi \circ \mathring{R})(\vp),\psi \rangle =(\pi \circ \mathring{R})(\vp,\psi)
\end{equation*}
for $\vp,\psi \in S^2_0(V)$. Therefore, the spectrum of the curvature operator of the second kind $\mathring{R}$ (as a bilinear form) is the same as the spectrum of the self-adjoint operator $\pi \circ \mathring{R}$.

First of all, we observe that \eqref{eq R product} implies that zero is an eigenvalue of $\mathring{R}$ with multiplicity (at least) $n_1n_2$. This is also observed in \cite[Lemma 2.1]{NPWW22}.
\begin{lemma}\label{lemma 3.1}
Let $R_i \in S^2_B(\Lambda^2 V_i)$ for $i=1,2$ with $\dim(V_i)=n_i \geq 1$ and let $R=R_1 \oplus R_2$. Let $E$ be the subspace of $S^2_0(V_1\times V_2)$ given by 
\begin{equation*}
    E=\spn \{u \odot v: u \in V_1, v\in V_2\}
\end{equation*}
Then $E$ lies in the kernel of $\mathring{R}$. 
In particular, $0$ is an eigenvalue of $\mathring{R}$ with multiplicity (at least) $n_1n_2$. 
\end{lemma}
\begin{proof}
We start by constructing an orthonormal basis of $E$. 
Let $\{e_i\}_{i=1}^{n_1}$ be an orthonormal basis of $V_1$ and $\{e_i\}_{i=n_1+1}^{n_1+n_2}$ be an orthonormal basis of $V_2$.
Then $\{e_i\}_{i=1}^{n_1+n_2}$ is an orthonormal basis of $V=V_1 \times V_2$. 
Define 
\begin{equation*}
    \xi_{pq}=\frac{1}{\sqrt{2}} e_p \odot e_q,
\end{equation*}   
for $1\leq p \leq n_1$ and $n_1+1 \leq q \leq n_1+n_2$. 
Then one can verify that the $\xi_{pq}$'s are traceless symmetric two-tensors on $V_1\times V_2$ and they form an orthonormal basis of $E$. In particular, $\dim(E)=n_1n_2$. 

To prove that $E$ lies in the kernel of $\mathring{R}$, it suffices to show that $\mathring{R}(\xi_{pq})=0$. We first observe that \eqref{eq R product} implies that
\begin{equation}\label{eq R=R1+R_2}
    R(e_i,e_j,e_k,e_l)=
    \begin{cases}
        R_1(e_i,e_j,e_k,e_l), & i,j,k,l \in \{1, \cdots, n_1\}, \\
        R_2(e_i,e_j,e_k,e_l), & i,j,k,l \in \{n_1+1, \cdots, n_1+n_2\}, \\
        0, & \text{otherwise. }
    \end{cases}
\end{equation}
We then compute, using $(e_p \odot e_q)(e_j,e_k)=(\delta_{pj}\delta_{qk}+\delta_{qj}\delta_{pk})$, that
\begin{eqnarray*}
\mathring{R}(\xi_{pq})(e_i,e_l) &=& \sum_{j,k=1}^n R(e_i,e_j,e_k,e_l)\xi_{pq}(e_j,e_k)\\
&=&\frac{1}{\sqrt{2}} \sum_{j,k=1}^{n}  R(e_i,e_j,e_k,e_l)(\delta_{pj}\delta_{qk}+\delta_{qj}\delta_pk) \\
&=& \frac{1}{\sqrt{2}} \sum_{j,k=1}^{n_1} \left( R(e_i,e_p,e_q,e_l)+R(e_i,e_q,e_p,e_l) \right) \\
&=& 0,
\end{eqnarray*}
where the last step is because of \eqref{eq R=R1+R_2} and the fact that $1\leq p \leq n_1$ and  $n_1+1 \leq q \leq n_1+n_2$.
Thus we have proved that $0$ is an eigenvalue of $\mathring{R}$ with multiplicity (at least) $n_1n_2$. 

\end{proof}

Next, we show that the eigenvalues of $R_1$ and $R_2$ are also eigenvalues of $R=R_1\oplus R_2$, provided that both $R_1$ and $R_2$ are Einstein. 
\begin{lemma}\label{lemma 3.2}
Let $R_i \in S^2_B(\Lambda^2 V_i)$ for $i=1,2$ with $\dim(V_i)=n_i \geq 1$ and let $R=R_1 \oplus R_2$. 
If $R_1$ (respectively, $R_2$) is Einstein, then the eigenvalues of $\mathring{R}_1$ (respectively, $\mathring{R}_2$) are also eigenvalues of $\mathring{R}$.
\end{lemma}
\begin{proof}
It suffices to prove the statement for $R_1$. Since $R_1$ is Einstein, we have that $\mathring{R}_1: S^2_0(V_1) \to S^2_0(V_1)$ is a self-adjoint operator. We can then choose an orthonormal basis $\{ \vp_p \}_{p=1}^{N_1}$ of $S^2_0(V_1)$ such that $$\mathring{R}_{1}(\vp_p)=\lambda_p \vp_p,$$ 
where $N_1=\frac{(n_1-1)(n_1+2)}{2}$ is the dimension of $S^2_0(V_1)$. 
%and $\{\psi_q \}_{q=1}^{N_2}$ be an orthonormal basis of $S^2_0(V_2)$ such that $\mathring{R}_{2}(\psi_q)=\mu_q \psi_q$, where $N_i=\frac{(n_i-1)(n_i+2)}{2}$ is the dimension of $S^2_0(V_i)$ for $i=1,2$.
Note that we may also view the $\vp_p$'s as elements in $S^2_0(V_1\times V_2)$ via zero extension, namely, 
\begin{align*}
    \vp_p(X_1+X_2,Y_1+Y_2) &=\vp_p(X_1,Y_1),
\end{align*}
for $X_i,Y_i \in V_i$. Then we have
\begin{equation}\label{eq vp extension}
    \vp_p(e_j,e_k) =
    \begin{cases}
        \vp_p(e_j,e_k), & j,k\in \{1,\cdots, n_1\}, \\
        0, & \text{otherwise},
    \end{cases}
\end{equation}
where $\{e_i\}_{i=1}^{n_1+n_2}$ is the same basis of $V$ in Lemma \ref{lemma 3.1}.

Next, we calculate using \eqref{eq vp extension} that for $1\leq i,l \leq n_1$, 
\begin{eqnarray*}
\mathring{R}(\vp_p)(e_i,e_l) &=& \sum_{j,k=1}^{n_1+n_2} R(e_i,e_j,e_k,e_l)\vp_p(e_j,e_k)\\
&=& \sum_{j,k=1}^{n_1}  R(e_i,e_j,e_k,e_l)\vp_p(e_j,e_k) \\
&=& \sum_{j,k=1}^{n_1}  R_1(e_i,e_j,e_k,e_l)\vp_p(e_j,e_k) \\
&=& \l_p \vp_p (e_i,e_l),
\end{eqnarray*}
and for $n_1+1\leq i,l \leq n_1+n_2$, 
\begin{eqnarray*}
\mathring{R}(\vp_p)(e_i,e_l) &=& \sum_{j,k=1}^{n_1+n_2} R(e_i,e_j,e_k,e_l)\vp_p(e_j,e_k)\\
&=& \sum_{j,k=1}^{n_1}  R(e_i,e_j,e_k,e_l)\vp_p(e_j,e_k) \\
%&=& \sum_{j,k=1}^{n_1}  R_1(e_i,e_j,e_k,e_l)\vp_p(e_j,e_k) \\
%&=& \l_p \vp_p (e_i,e_l) \\
&=& 0 \\
&=& \l_p \vp_p (e_i,e_l).
\end{eqnarray*}
Therefore, we have proved $\mathring{R}(\vp_p)=\l_p \vp_p$ for $1\leq p\leq N_1$. 
Hence the eigenvalues of $\mathring{R}_1$ are also eigenvalues of $\mathring{R}$ with the same eigenvectors. 
%A similar argument shows that the eigenvalues of $\mathring{R}_2$ are also eigenvalues of $\mathring{R}$ with the same eigenvectors. 
%A similar argument yields $\mathring{R}(\psi_q)=\mu_q \psi_q$ for $1\leq q\leq N_2$. 
\end{proof}

Finally, we prove that  
\begin{lemma}\label{lemma 3.3}
Let $R_i \in S^2_B(\Lambda^2 V_i)$ for $i=1,2$ with $\dim(V_i)=n_i \geq 1$ and let $R=R_1 \oplus R_2$. 
If $\Ric(R_i)=\rho_i g_i$ for $i=1,2$, then $-\frac{n_2 \rho_1 +n_1 \rho_2}{n_1+n_2}$ is an eigenvalue of $\mathring{R}$ with eigenvector $n_2 g_1 - n_1 g_2$.
\end{lemma}

\begin{proof}
As in the proof of Lemma \ref{lemma 3.2}, we may also view $g_1$ and $g_2$ as elements in $S^2_0(V_1\times V_2)$ via zero extension. 
Clearly, $\tr(n_2 g_1 - n_1 g_2)=n_2n_1-n_1n_2=0$. So we have $n_2 g_1 - n_1 g_2 \in S^2_0(V_1\times V_2)$. 

We then compute that
\begin{eqnarray*}
\mathring{R}(n_2g_1-n_1g_2)
&=& n_2 \mathring{R}(g_1) -n_1 \mathring{R}(g_2)  \\ 
&=& n_2 \mathring{R}_1(g_1) -n_1 \mathring{R}_2(g_2)  \\ 
&=& -n_2 \Ric(R_1)+n_1 \Ric(R_2) \\ 
&=& -n_2 \rho_1 g_1 +n_1 \rho_2 g_2,
\end{eqnarray*}
where we have used $\mathring{R}_i(g_i)=-\Ric(R_i)=-\rho_i g_i$ for $i=1,2$. 

Using 
$$\tr(-n_2 \rho_1 g_1 +n_1 \rho_2 g_2)=-n_1n_2(\rho_1-\rho_2),$$ 
we then obtain that
\begin{eqnarray*}
&& (\pi \circ \mathring{R})(n_2g_1-n_1g_2) \\
&=& -n_2 \rho_1 g_1 +n_1 \rho_2 g_2 - \frac{-n_1n_2(\rho_1-\rho_2)}{n_1+n_2}(g_1+g_2) \\
&=& -n_2g_1 \left(\rho_1 -\frac{n_1(\rho_1-\rho_2)}{n_1+n_2} \right) +n_1g_2 \left( \rho_2 +\frac{n_2(\rho_1-\rho_2)}{n_1+n_2} \right) \\
&=& -\left(\frac{n_1\rho_2+n_2\rho_1 }{n_1+n_2} \right)  (n_2g_1-n_1g_2).
\end{eqnarray*}
Thus, we see that $-\frac{n_1\rho_2+n_2\rho_1 }{n_1+n_2}$ is an eigenvalue of $\mathring{R}$ with eigenvector $n_2g_1-n_1g_2$. The proof is now complete. 
\end{proof}

\begin{proof}[Proof of Proposition \ref{prop 3.5}]
Let $\{e_i\}_{i=1}^{n_1+n_2}$ be an orthonormal basis of $V$ with $e_1, \cdots e_{n_1} \in V_1$ and $e_{n_1+1}, \cdots e_{n_1+n_2} \in V_2$. 
Let $\{ \vp_p \}_{p=1}^{N_1}$ be an orthonormal basis of $S^2_0(V_1)$ such that $\mathring{R}_{1}(\vp_p)=\lambda_p \vp_p$ and $\{\psi_q \}_{q=1}^{N_2}$ be an orthonormal basis of $S^2_0(V_2)$ such that $\mathring{R}_{2}(\psi_q)=\mu_q \psi_q$, where $N_i=\frac{(n_i-1)(n_i+2)}{2}$ is the dimension of $S^2_0(V_i)$ for $i=1,2$.
We then define the following traceless symmetric two-tensors on $V$:
\begin{equation*}
    \xi_{pq}=\frac{1}{\sqrt{2}} e_p \odot e_q
\end{equation*}
for $1\leq p \leq n_1$ and  $n_1+1 \leq q \leq n_1+n_2$, and 
\begin{equation*}
    \zeta=\frac{1}{\sqrt{n_1n_2(n_1+n_2)}} \left( n_2 g_1 - n_1 g_2 \right).
\end{equation*}
Then one can verify, via straightforward computations, that 
$$\{\vp_p\}_{p=1}^{N_1} \cup \{\psi_q\}_{q=1}^{N_2} \cup \{\xi_{pq}\}_{1\leq p \leq n_1, n_1+1\leq q \leq n_1+n_2} \cup \{\zeta\} $$ form an orthonormal basis of $S^2_0(V)$. 

According to Lemma \ref{lemma 3.1}, \ref{lemma 3.2} and \ref{lemma 3.3}, the above basis diagonalizes $\mathring{R}$ as 
\begin{equation*}
\begin{pmatrix}
\begin{matrix}
  \lambda_1  \\
  & \ddots & \\
  & & \lambda_{N_1}
 \end{matrix}
  & \rvline &  & \rvline &  & \rvline  &  \\
\hline  & \rvline & 
  \begin{matrix}
  \mu_1  \\
  & \ddots & \\
  & & \mu_{N_2}
  \end{matrix}
  & \rvline &   & \rvline &  \\
 \hline  & \rvline &  & \rvline & 
  \begin{matrix}
  0 &  \\
  & \ddots & \\
  &  &  0
  \end{matrix}
  & \rvline &   & \\
  \hline  & \rvline &  & \rvline &  & \rvline & 
\begin{matrix}
-\frac{n_2 \rho_1 +n_1 \rho_2}{n_1+n_2}
\end{matrix}
\end{pmatrix}
\end{equation*}

\end{proof}

Theorem \ref{thm product eigenvalue} now follows immediately from Proposition \ref{prop 3.5}, since on a product manifold the product metric satisfies \eqref{eq g product} and the Riemann curvature tensor satisfies \eqref{eq R product}. 

Since the spectrum of $\mathring{R}$ are known on space forms with constant sectional curvature and K\"ahler space forms with constant holomorphic sectional curvature, we can use Theorem \ref{thm product eigenvalue} or Proposition \ref{prop 3.5} to determine the eigenvalues of the curvature operator of the second kind on their product. 

In the rest of this section, 
\begin{itemize}
    \item $\S^n(\kappa)$ and $\H^n(-\kappa)$, $n\geq 2$ and $\kappa >0$,  denote the $n$-dimensional simply-connected space form with constant sectional curvature $\kappa$ and $-\kappa$, respectively.
    \item $\CP^m(\kappa)$ and $\CH^m(-\kappa)$, $m\geq 1$ and $\kappa>0$,  denote the (complex) $m$-dimensional simply-connected K\"ahler space form with constant holomorphic sectional curvature $4\kappa$ and $-4\kappa$, respectively. 
\end{itemize}

\begin{example}
$\mathring{R}=\kappa \id_{S^2_0}$ on $\S^n(\kappa)$. $\mathring{R}=-\kappa \id_{S^2_0}$ on $\H^n(-\kappa)$. 
\end{example}
\begin{example}
$\mathring{R}$ has two distinct eigenvalues on $\mathbb{CP}^m(\kappa)$: $-2\kappa$ with multiplicity $(m-1)(m+1)$ and $4\kappa$ with multiplicity $m(m+1)$. 
$\mathring{R}$ has two distinct eigenvalues on $\mathbb{CH}^m(-\kappa)$: $2\kappa$ with multiplicity $(m-1)(m+1)$ and $-4\kappa$ with multiplicity $m(m+1)$. See \cite{BK78}. 
\end{example}
\begin{example}
Let $M=\S^{n_1}(\kappa_1)\times \S^{n_2}(\kappa_2)$. Then the curvature operator of the second kind of $M$ has eigenvalues: $\k_1$ with multiplicity $\frac{(n_1-1)(n_1+2)}{2}$, $\k_2$ with multiplicity $\frac{(n_2-1)(n_2+2)}{2}$, $0$ with multiplicity $n_1n_2$ and $-\frac{n_1(n_2-1)\k_2+n_2(n_1-1)\k_1}{n_1+n_2}$ with multiplicity one.
\end{example}

\begin{example}
Let $M=\H^{n_1}(-\k_1)\times \H^{n_2}(-\k_2)$. Then the curvature operator of the second kind of $M$ has eigenvalues: $-\k_1$ with multiplicity $\frac{(n_1-1)(n_1+2)}{2}$, $-\k_2$ with multiplicity $\frac{(n_2-1)(n_2+2)}{2}$, $0$ with multiplicity $n_1n_2$ and $\frac{n_1(n_2-1)\k_2+n_2(n_1-1)\k_1}{n_1+n_2}$ with multiplicity one.
\end{example}

\begin{example}
Let $M=\S^{n_1}(\k_1)\times \R^{n_2}$. Then the curvature operator of the second kind of $M$ has eigenvalues: $\k_1$ with multiplicity $\frac{(n_1-1)(n_1+2)}{2}$, $0$ with multiplicity $n_1n_2+\frac{(n_2-1)(n_2+2)}{2}$ and $-\frac{n_2(n_1-1)\k_1}{n_1+n_2}$ with multiplicity one.
\end{example}

\begin{example}
Let $M=\H^{n_1}(-\k_1)\times \R^{n_2}$. Then the curvature operator of the second kind of $M$ has eigenvalues: $-\k_1$ with multiplicity $\frac{(n_1-1)(n_1+2)}{2}$, $0$ with multiplicity $n_1n_2+\frac{(n_2-1)(n_2+2)}{2}$ and $\frac{n_2(n_1-1)\k_1}{n_1+n_2}$ with multiplicity one.
\end{example}

\begin{example}
Let $M=\S^{n_1}(\k_1)\times \H^{n_2}(-\k_2)$. Then the curvature operator of the second kind of $M$ has eigenvalues: $\k_1$ with multiplicity $\frac{(n_1-1)(n_1+2)}{2}$, $-\k_2$ with multiplicity $\frac{(n_2-1)(n_2+2)}{2}$, $0$ with multiplicity $n_1n_2$ and $-\frac{n_1n_2(\k_1-\k_2)+n_1\k_2-n_2\k_1}{n_1+n_2}$ with multiplicity one.
\end{example}

\begin{example}
Let $M=\CP^{m_1}(\k_1)\times \CP^{m_2}(\k_2)$. Then the curvature operator of the second kind of $M$ has eigenvalues:  $-2\k_1$ with multiplicity $(m_1-1)(m_1+1)$, $-2\k_2$ with multiplicity $(m_2-1)(m_2+1)$, $4\k_1$ with multiplicity $m_1(m_1+1)$, $4\k_2$ with multiplicity $m_2(m_2+1)$, $0$ with multiplicity $4m_1m_2$, and $-\frac{2m_1(m_2+1)\k_2 +2m_2(m_1+1)\k_1}{m_1+m_2}$ with multiplicity one. 
\end{example}

\begin{example}
Let $M=\CH^{m_1}(-\k_1)\times \CH^{m_2}(-\k_2)$. Then the curvature operator of the second kind of $M$ has eigenvalues:  $2\k_1$ with multiplicity $(m_1-1)(m_1+1)$, $2\k_2$ with multiplicity $(m_2-1)(m_2+1)$, $-4\k_1$ with multiplicity $m_1(m_1+1)$, $-4\k_2$ with multiplicity $m_2(m_2+1)$, $0$ with multiplicity $4m_1m_2$, and $\frac{2m_1(m_2+1)\k_2 +2m_2(m_1+1)\k_1}{m_1+m_2}$ with multiplicity one. 
\end{example}

\begin{example}
Let $M=\CP^{m_1}(\k_1)\times \C^{m_2}$. Then the curvature operator of the second kind of $M$ has eigenvalues:  $-2\k_1$ with multiplicity $(m_1-1)(m_1+1)$, $4\k_1$ with multiplicity $m_1(m_1+1)$, $0$ with multiplicity $4m_1m_2+(2m_2-1)(m_2+1)$, and $-\frac{2m_2(m_1+1)\k_1}{m_1+m_2}$ with multiplicity one. 
\end{example}

\begin{example}
Let $M=\CH^{m_1}(-\k_1)\times \C^{m_2}$. Then the curvature operator of the second kind of $M$ has eigenvalues:  $2\k_1$ with multiplicity $(m_1-1)(m_1+1)$, $-4\k_2$ with multiplicity $m_1(m_1+1)$, $0$ with multiplicity $4m_1m_2+(2m_2-1)(m_2+1)$, and $\frac{2m_2(m_1+1)\k_1}{m_1+m_2}$ with multiplicity one. 
\end{example}

\begin{example}
Let $M=\CP^{m_1}(\k_1)\times \CH^{m_2}(-\k_2)$. Then the curvature operator of the second kind of $M$ has eigenvalues:  $-2\k_1$ with multiplicity $(m_1-1)(m_1+1)$, $4\k_2$ with multiplicity $m_1(m_1+1)$, $2\k_2$ with multiplicity $(m_2-1)(m_2+1)$, $-4\k_2$ with multiplicity $m_2(m_2+1)$, $0$ with multiplicity $4m_1m_2$, and $-\frac{2m_1m_2(\k_1-\k_2)+2m_2\k_1 -2m_1\k_2}{m_1+m_2}$ with multiplicity one. 
\end{example}

In particular, we have the following observation, which will be needed later on. 
\begin{proposition}\label{prop k1=k2}
For $n_1,n_2 \geq 2$, $m_1,m_2 \geq 1$, $\k_1, \k_2 >0$, the following statements hold:
\begin{enumerate}
    \item $\S^{n_1}(\kappa_1)\times \S^{n_2}(\kappa_2)$ has $A_{n_1,n_2}$-nonnegative curvature operator of the second kind if and only if $\k_1=\k_2>0$;
    \item $\H^{n_1}(-\kappa_1)\times \H^{n_2}(-\kappa_2)$ has $A_{n_1,n_2}$-nonpositive curvature operator of the second kind if and only if $\k_1=\k_2>0$;
    \item $\CP^{m_1}(\k_1)\times \CP^{m_2}(\k_2)$ has $B_{m_1,m_2}$-nonnegative curvature operator of the second kind if and only if $\k_1=\k_2>0$;
    \item $\CH^{m_1}(-\k_1)\times \CH^{m_2}(-\k_2)$ has $B_{m_1,m_2}$-nonpositive curvature operator of the second kind if and only if $\k_1=\k_2 < 0$.
\end{enumerate}
\end{proposition}
\section{Rigidity of $\S^{n-1}\times \R$ and $\H^{n-1}\times \R$}

In this section, we prove Theorem \ref{thm SnXR}.
The key result of this section is the following proposition. 
\begin{proposition}\label{prop 4.3}
Let $(V,g)$ be a Euclidean vector space of dimension $(n-1)$ with $n \geq 2$ and let $R_1 \in S^2_B(\Lambda^2 V)$. 
\begin{enumerate}
    \item Suppose that $R=R_1\oplus 0 \in S^2_B(\Lambda^2 (V \times \R))$ has $(n+\frac{n-2}{n})$-nonnegative curvature operator of the second kind. Then $R_1$ has constant nonnegative sectional curvature. 
    %=c \I_{n-1} $ for some $c\geq 0$. 
    \item Suppose that $R=R_1\oplus 0 \in S^2_B(\Lambda^2 (V \times \R))$ has $(n+\frac{n-2}{n})$-nonpositive curvature operator of the second kind. Then $R_1$ has constant nonpositive sectional curvature. 
    \item Suppose that $R=R_1\oplus 0 \in S^2_B(\Lambda^2 (V \times \R))$ has $\a$-nonnegative or $\a$-nonpositive curvature operator of the second kind for some $\a < n+\frac{n-2}{n}$, then $R$ is flat.
    %=c \I_{n-1}$ for some $c\leq 0$. 
\end{enumerate} 
\end{proposition}

%\begin{proposition}\label{prop 4.1}
%Let $(M^n,g)$, $n\geq 4$, be a nonflat complete locally reducible Riemannian manifold. 
%\begin{enumerate}
%    \item If $M$ has $(n+\frac{n-2}{n})$-nonnegative curvature operator of the second kind, then the universal cover of $M$ is, up to scaling, isometric to the $\mathbb{S}^{n-1}\times \R$. 
%    \item If $M$ has $(n+\frac{n-2}{n})$-nonpositive curvature operator of the second kind, then the universal cover of $M$ is, up to scaling, isometric to the $\mathbb{H}^{n-1}\times \R$. 
%\end{enumerate}
%\end{proposition}

\begin{proof}
(1). Let $\{e_i\}_{i=1}^{n-1}$ be an orthonormal basis of $V$ and let $e_n$ be a unit vector in $\R$. 
Then $\{e_i\}_{i=1}^{n}$ is an orthonormal basis of $V\times \R \cong V\oplus \R$.  
Next, we define the following symmetric two-tensors on $V\oplus \R$:
\begin{eqnarray*}
\xi_i &=& \frac{1}{\sqrt{2}} e_i\odot e_n \text{ for } 1 \leq i \leq n-1, \\
\vp_{kl} &=& \frac{1}{\sqrt{2}}e_k\odot e_l \text{ for }  1 \leq k < l \leq n-1, \\
\zeta &=& \frac{1}{2\sqrt{n(n-1)}} \left( \sum_{p=1}^{n-1} e_p\odot e_p -(n-1) e_n \odot e_n  \right).  
\end{eqnarray*}
One easily verifies that $\{\xi_{i}\}_{i=1}^{n-1} \cup \{\vp_{kl}\}_{1 \leq k < l \leq n-1} \cup \{\zeta\}$ form an orthonormal subset of $S^2_0(\Lambda^2 (V\oplus \R))$. 

Since $R=R_1 \oplus 0$, we have by \eqref{eq R product} that
\begin{equation}\label{eq R=R1+0}
    R(e_i,e_j,e_k,e_l)=
    \begin{cases}
        R_1(e_i,e_j,e_k,e_l), & i,j,k,l \in \{1, \cdots, n-1\}, \\
        %R_2(e_i,e_j,e_k,e_l), & i,j,k,l \in \{n_1+1, \cdots, n_1+n_2\}, \\
        0, & \text{otherwise. }
    \end{cases}
\end{equation}
In particular, we have $R_{njnj}=0$ for $1\leq j \leq n-1$. 

Direct calculation using the identity
\begin{equation*}\label{eq 3.1}
        \mathring{R}(e_i \odot e_j ,e_k \odot e_l)= 2(R_{iklj}+R_{ilkj})
    \end{equation*}
shows that 
\begin{eqnarray*}
\mathring{R}(\xi_i,\xi_i) &=& 0 \text{ for } 1\leq i \leq n-1, \\
\mathring{R}(\vp_{kl},\vp_{kl}) &=& (R_1)_{klkl} \text{ for } 1 \leq k < l \leq n-1, \\
\mathring{R}(\zeta,\zeta) &=& -\frac{1}{n(n-1)}S_1, 
\end{eqnarray*}
where $S_1$ is the scalar curvature of $R_1$. Note that $S_1 \geq 0$ since $S_1$ is also equal to the scalar curvature of $R$, which must be nonnegative since $R$ has $(n+\frac{n-2}{n})$-nonnegative curvature operator of the second kind (see for instance \cite[Proposition 4.1, part (1)]{Li21}). 

Since $R$ has $(n+\frac{n-2}{n})$-nonnegative curvature operator of the second kind, we get that for any $ 1 \leq k < l \leq n-1$, 
\begin{eqnarray*}
0 & \leq & \mathring{R}(\zeta,\zeta) + \sum_{i=1}^{n-1} \mathring{R}(\xi_i,\xi_i) + \frac{n-2}{n} \mathring{R}(\psi_{kl},\psi_{kl}) \\
&=& -\frac{1}{n(n-1)}S_1 + \frac{n-2}{n} (R_1)_{klkl} \\
&=& \frac{n-2}{n} \left((R_1)_{klkl} -\frac{S_1}{(n-1)(n-2)} \right).
\end{eqnarray*}
Summing over $1\leq k < l \leq n-1$ yields 
\begin{equation*}
    S_1 \leq \sum_{1\leq k < l \leq n-1} (R_1)_{klkl}.
\end{equation*}
On the other hand, 
$$S_1=\sum_{1\leq k < l \leq n-1} (R_1)_{klkl}.$$ 
Therefore, we must have $(R_1)_{klkl}=\frac{S_1}{(n-1)(n-2)}$ for all $1\leq k < l \leq n-1$. 
Since the orthonormal basis $\{e_1, \cdots, e_{n-1}\}$ is arbitrary, we conclude that $R_1$ has constant nonnegative sectional curvature. 

%Finally, we note that $S_1 \geq 0$ since $S_1$ is also equal to the scalar curvature of $R$, which must be nonnegative since $R$ has $(n+\frac{n-2}{n})$-nonnegative curvature operator of the second kind. 
%Thus we have proved that $R_1=c \I_{n-1}$ for some $c\geq 0$. 

(2). Apply (1) to $-R$. 

(3). By (1) and (2), we have $R=c I_{n-1}\oplus \ 0$ for some $c \in \R$, where $I_{n-1}$ is the Riemann curvature tensor of $\S^{n-1}$. However, $R=c I_{n-1}\oplus \ 0$ has $\a$-nonnegative or $\a$-nonpositive curvature operator of the second kind for some $\a < n+\frac{n-2}{n}$ if and only if $c=0$. Therefore, $R$ is flat.
\end{proof}

We now present the proof of Theorem \ref{thm SnXR}. 
\begin{proof}[Proof of Theorem \ref{thm SnXR}]
(1). Recall that we say that $(M^n,g)$ is locally reducible if there exists a nontrivial subspace of $T_pM$ which is invariant under the action of the restricted holonomy group.
By a theorem of de Rham, a complete Riemannian manifold is locally reducible
if and only if its universal cover is isometric to the product of two Riemannian manifolds of lower dimension. 

Denote by $(\widetilde{M},\tilde{g})$ the universal cover of $M$ with the lifted metric $\tilde{g}$. Since $M$ is locally reducible, $(\widetilde{M},\tilde{g})$ is isometric to a product of the form $(M_1^k,g_1) \times (M^{n-k}_2,g_2)$, where $1\leq k \leq \frac{n}{2}$. 
Note that $k\geq 2$ implies $$k(n-k)+1 \geq n+\frac{n-2}{n},$$
so $\widetilde{M}$ must be flat if $k\geq 2$, according to \cite[Proposition 5.1]{Li21} (or its improvement Theorem \ref{thm split}). 
Thus we must have $k=1$ and $\widetilde{M}$ is isometric to $N^{n-1} \times \R$. 
By part (1) of Proposition \ref{prop 4.3}, $N$ has pointwise constant nonnegative sectional curvature.     
Since $n-1\geq 3$, Schur's lemma\ implies that $N$ must have constant nonnegative sectional curvature. 
Therefore, $M$ is either flat or its universal cover is isometric to $\S^{n-1}\times \R$ by scaling. 

(2). This is similar to the proof of (1), by noticing that \cite[Proposition 5.1]{Li21} is also valid for the nonpositivity condition (alternatively, one can use Theorem \ref{thm split} here). 
\end{proof}

\begin{proof}[Proof of Theorem \ref{thm 4.5 nonnegative}]
Let $(M^n,g)$ be a closed non-flat Riemannian manifold of dimension $n\geq 4$ and suppose that $M$ has $4\frac 1 2$-nonnegative curvature operator of the second kind. 
It was shown in \cite{Li22JGA} that one of the following statements holds:
\begin{enumerate}
    \item[(a)] $M$ is homeomorphic (diffeomorphic if $n=4$ or $n\geq 12$) to a spherical space form;
    \item[(b)] $n=2m$ and the universal cover of $M$ is a K\"ahler manifold biholomorphic to $\mathbb{CP}^m$;
    \item[(c)] $n=4$ and the universal cover of $M$ is diffeomorphic to $\mathbb{S}^3 \times \R$; % or $M$ is a K\"ahler surface biholomorphic to $\mathbb{CP}^2$;
    \item[(d)] $n\geq 5$ and $M$ is isometric to a quotient of a compact irreducible symmetric space.
\end{enumerate}
By Theorem 1.2 in \cite{Li22Kahler}, the K\"ahler manifold in part (2) is either flat or isometric to $\CP^2$ with the Fubini-Study metric, up to scaling. 
In part (c), the manifold is reducible and we conclude using Theorem \ref{thm SnXR} that the universal cover of $M$ is isometric to $\S^3\times \R$, up to scaling. 
Part (d) can be ruled out using \cite[Theorem B]{NPW22}, as the manifold is either flat or a homology sphere.

\end{proof}

%An immediate corollary of Proposition \ref{prop 4.1} is 
%\begin{corollary}
%Let $(M^n,g)$ be a Riemannian manifold with $\alpha$-nonnegative or $\alpha$-nonpositive curvature operator of the second kind for some $\alpha > n+\frac{n-2}{n}$. Then $M$ is either flat or locally irreducible. 
%\end{corollary}

\section{Rigidity of $\S^{n_1}\times \S^{n_2}$ and $\H^{n_1}\times \H^{n_2}$ }

In this section, we prove Theorem \ref{thm product spheres}. The key result of this section is the following proposition. In this section, $I_n$, $n\geq 2$, denotes the Riemann curvature tensor of the $n$sphere with constant sectional curvature $1$. 
\begin{proposition}\label{prop 5.3}
For $i=1,2$, let $(V_i,g_i)$ be a Euclidean vector space of dimension $n_i$ with $n_i \geq 2$.  Let $R_i \in S^2_B(\Lambda^2 V_i)$ and $R=R_1\oplus R_2 \in S^2_B(\Lambda^2 (V_1 \times V_2))$. 
\begin{enumerate}
    \item Suppose that $R$ has $A_{n_1,n_2}$-nonnegative curvature operator of the second kind. Then $R=c(I_{n_1} \oplus I_{n_2})$ for some $c\geq 0$.
    \item Suppose that $R$ has $A_{n_1,n_2}$-nonpositive curvature operator of the second kind. Then $R=c(I_{n_1} \oplus I_{n_2})$ for some $c\leq 0$.
    \item Suppose that $R$ has $\a$-nonnegative or $\a$-nonpositive curvature operator of the second kind for some $\a < A_{n_1,n_2}$, then $R$ is flat.
\end{enumerate} 
\end{proposition}

We need an elementary lemma, which can be found in \cite[Lemma 5.1]{Li22Kahler}. 
\begin{lemma}\label{lemma average}
Let $N$ be a positive integer and $A$ be a collection of $N$ real numbers.  Denote by $a_i$ the $i$-th smallest number in $A$ for $1\leq i \leq N$.  
Define a function $f(A,x)$ by 
   \begin{equation*}
       f(A,x)=\sum_{i=1}^{\lfloor x \rfloor} a_i +(x-\lfloor x \rfloor) a_{\lfloor x \rfloor+1}, 
   \end{equation*}
   for $x\in [1,N]$. Then we have
   \begin{equation}\label{eq function}
       f(A,x) \leq x \bar{a}, 
   \end{equation}
   where $\bar{a}:=\frac{1}{N}\sum_{i=1}^N a_i$ is the average of all numbers in $A$. 
   Moreover, the equality holds for some $x\in [1,N)$ if and only if $a_i=\bar{a}$ for all $1\leq i\leq N$. 
\end{lemma}

\begin{proof}[Proof of Proposition \ref{prop 5.3}]
(1). Let $\{e_i\}_{i=1}^{n_1}$ be an orthonormal basis of $V_1$ and let $\{e_{i}\}_{i=n_1+1}^{n_1+n_2}$ be an orthonormal basis of $V_2$.
Then $\{e_i\}_{i=1}^{n_1+n_2}$ is an orthonormal basis of $V_1 \times V_2 \cong V_1\oplus V_2$.  

We construct an orthonormal basis of $S^2_0(V_1\times V_2)$ as follows. 
Choose an orthonormal basis $\{\vp_i\}_{i=1}^{N_1}$ of $S^2_0(V_1)$ and an orthonormal basis $\{\psi_i\}_{i=1}^{N_2}$ of $S^2_0(V_2)$, where $N_i=\dim(S^2_0(V_i))=\frac{(n_i-1)(n_i+2)}{2}$ for $i=1,2$.
Note that $h \in S^2_0(V_1)$ can be identified with the element $\pi^* h$ in $S^2_0(V_1\times V_2)$ via 
\begin{equation*}
    (\pi^* h)(X_1+X_2,Y_1+Y_2)=h(X_1,X_2),
\end{equation*}
where $X_i,Y_i \in V_i$ for $i=1,2$. 
We shall simply write $\pi^* h$ as $h$. 
Similarly, $S^2_0(V_2)$ can be identified with a subspace of $S^2_0(V_1\times V_2)$. Next we define the following symmetric two-tensors on $V_1\times V_2$:
\begin{eqnarray*}
%\xi_i &=& \frac{1}{\sqrt{2}} e_i\odot e_n \text{ for } 1 \leq i \leq n-1, \\
\xi_{kl} &=& \frac{1}{\sqrt{2}}e_k\odot e_l \text{ for }  1 \leq k \leq n_1,  n_1+1\leq l \leq n_1+n_2, \\
\zeta &=& \frac{1}{\sqrt{n_1n_2(n_1+n_2)}} \left( n_2g_1 -n_1g_2\right).  
\end{eqnarray*}
One verifies that 
$$\{\vp_i\}_{i=1}^{N_1} \cup \{\psi_i\}_{i=1}^{N_2} \cup \{\xi_{kl}\}_{1 \leq k \leq n_1,  n_1+1\leq l \leq n_1+n_2} \cup \{\zeta\}$$
form an orthonormal basis of $S^2_0(V_1 \times V_2)$. This corresponds to the orthogonal decomposition 
\begin{equation*}
    S^2_0(V_1 \times V_2 )=S^2_0(V_1) \oplus S^2_0(V_2) \oplus \spn\{ u \odot v, u\in V_1, v\in V_2 \} \oplus \R \zeta.
\end{equation*}

The next step is to calculate some diagonal elements of the matrix representing $\mathring{R}$ with respect to the above basis. 
Since $R=R_1 \oplus R_2$, we have by \eqref{eq R product} that
\begin{equation}\label{eq R=R1+0}
    R(e_i,e_j,e_k,e_l)=
    \begin{cases}
        R_1(e_i,e_j,e_k,e_l), & i,j,k,l \in \{1, \cdots, n_1\}, \\
        R_2(e_i,e_j,e_k,e_l), & i,j,k,l \in \{n_1+1, \cdots, n_1+n_2\}, \\
        0, & \text{otherwise. }
    \end{cases}
\end{equation}
In particular, we have $R_{klkl}=0$ if $1\leq k \leq n_1$ and $n_1\leq l \leq n_1+n_2$. 
Using the identity
\begin{equation*}\label{eq 3.1}
        \mathring{R}(e_i \odot e_j ,e_k \odot e_l)= 2(R_{iklj}+R_{ilkj}),
\end{equation*}
we get
\begin{equation}\label{eq product zero}
    \sum_{\substack{1\leq k \leq n_1, \\ n_1+1\leq l \leq n_1+n_2}} \mathring{R}(\xi_{kl}, \xi_{kl})=\sum_{\substack{1\leq k \leq n_1, \\ n_1+1\leq l \leq n_1+n_2}} R_{klkl}=0. 
\end{equation}
We also calculate 
\begin{eqnarray*}
\mathring{R}(\zeta,\zeta) &=& \frac{1}{n_1n_2(n_1+n_2)} \left( n_2^2\mathring{R}(g_1,g_1) +n_1^2  \mathring{R}(g_2,g_2) +2n_1n_2 \mathring{R}(g_1,g_2)    \right)\\
&=& \frac{1}{n_1n_2(n_1+n_2)} \left( n_2^2\mathring{R_1}(g_1,g_1) +n_1^2  \mathring{R_2}(g_2,g_2)    \right)\\
&=& -\frac{n_2^2 S_1 +n_1^2S_2}{n_1n_2(n_1+n_2)} ,\\
\end{eqnarray*}
where $S_i$ denotes the scalar curvature of $R_i$ for $i=1,2$. 

Let $A$ be the collection of the values of $\mathring{R}(\vp_i,\vp_i)$ for $1\leq i \leq N_1$ and let $B$ be the collection of the values of $\mathring{R}(\psi_i,\psi_i)$ for $1\leq i \leq N_2$. 
Denote by $\bar{a}$ and $\bar{b}$ the average of all numbers is $A$ and $B$, respectively. Then 
\begin{eqnarray*}
    && \bar{a} =\frac{1}{N_1} \sum_{i=1}^{N_1} \mathring{R}(\vp_i,\vp_i) 
     =\frac{1}{N_1} \sum_{i=1}^{N_1} \mathring{R_1}(\vp_i,\vp_i)
    = \frac{S_1}{n_1(n_1-1)}, \\
    && \bar{b}=\frac{1}{N_2} \sum_{i=1}^{N_2} \mathring{R}(\psi_i,\psi_i) 
   = \frac{1}{N_2} \sum_{i=1}^{N_2} \mathring{R_2}(\psi_i,\psi_i)
    =\frac{S_2}{n_2(n_2-1)},
\end{eqnarray*}
where we have used
\begin{eqnarray*}
      \sum_{i=1}^{N_1} \mathring{R_1}(\psi_i,\psi_i)
    = \frac{n_1+2}{2n_1}S_1 
    \text{ and } \sum_{i=1}^{N_2} \mathring{R_2}(\psi_i,\psi_i)
    = \frac{n_2+2}{2n_2}S_2. 
\end{eqnarray*}

For simplicity, we write 
\begin{equation*}
    A_1=\frac{n_2(n_1-1)}{n_1+n_2} \text{ and }  A_2=\frac{n_1(n_2-1)}{n_1+n_2}.
\end{equation*}
Notice that we have $A_1<N_1$, $A_2<N_2$ and
\begin{equation}\label{eq A_n_1n_2}
    A_{n_1,n_2} = 1+n_1n_2+A_1+A_2.
\end{equation}
Also, the expression for $\mathring{R}(\zeta,\zeta)$ can be written as
\begin{equation}\label{eq R zeta 1}
    \mathring{R}(\zeta,\zeta)=-A_1\bar{a}-A_2\bar{b}.
\end{equation}

%Next, we divide the proof into two cases. % depending on whether $\lfloor N \rfloor=\lfloor A_1 \rfloor+\lfloor A_2 \rfloor$ or $\lfloor N \rfloor=\lfloor A_1 \rfloor+\lfloor A_2 \rfloor +1$. 

%\textbf{Case 1:} $\lfloor A_1+A_2 \rfloor=\lfloor A_1 \rfloor+\lfloor A_2 \rfloor$.
Since $R$ has $A_{n_1,n_2}$-nonnegative curvature operator of the second kind, we get using \eqref{eq product zero}, \eqref{eq A_n_1n_2} and \eqref{eq R zeta 1} that 
%\begin{equation*}
%     0 \leq \mathring{R}(\zeta,\zeta)+f(A, \lfloor A_1 \rfloor) + f(B, A_1+A_2-\lfloor A_1 \rfloor),
%\end{equation*}
\begin{eqnarray}\label{eq key 1}
    -\mathring{R}(\zeta,\zeta) &\leq& f(A, \lfloor A_1 \rfloor) + f(B, A_1+A_2-\lfloor A_1 \rfloor) \\ \nonumber
    & \leq & \lfloor A_1 \rfloor \bar{a} +(A_1+A_2-\lfloor A_1 \rfloor)\bar{b} \\ \nonumber
    &=& A_1 \bar{a} +A_2 \bar{b} +(A_1-\lfloor A_1 \rfloor) (\bar{b}-\bar{a}),
\end{eqnarray}
where $f$ is the function defined in Lemma \ref{lemma average} and we have used Lemma \ref{lemma average} in estimating $f$. 
Similarly, we also have 
\begin{eqnarray}\label{eq key 2}
    -\mathring{R}(\zeta,\zeta) &\leq& f(A, A_1+A_2 -\lfloor A_2 \rfloor) + f(B, \lfloor A_2 \rfloor), \\ \nonumber
    & \leq & (A_1+A_2 -\lfloor A_2 \rfloor) \bar{a} +\lfloor A_2 \rfloor) \bar{b} \\ \nonumber
    &=& A_1 \bar{a} +A_2 \bar{b} +(A_2-\lfloor A_2 \rfloor) (\bar{a}-\bar{b}).
\end{eqnarray}
Therefore, we get from \eqref{eq key 1} if $\bar{a}\geq \bar{b}$ and from \eqref{eq key 2} if  $\bar{a}\leq \bar{b}$ that 
\begin{equation*}
    A_1 \bar{a} +A_2 \bar{b} = -\mathring{R}(\zeta,\zeta) \leq A_1 \bar{a} +A_2 \bar{b}.
\end{equation*}
This implies that, either in \eqref{eq key 1} or \eqref{eq key 2}, we must have equalities in the inequalities used for $f$. We then get from Lemma \ref{lemma average}, that all the values in $A$ are equal to $\bar{a}$ and all the values in $B$ are equal to $\bar{b}$.  
Hence, both $R_1$ and $R_2$ have constant sectional curvature, that is to say, $R=c_1 I_{n_1} \oplus c_2 I_{n_2}$ for $c_1, c_2 \in \R$. 

Finally, we must have $c_1=c_2\geq 0$, as $R=c_1 I_{n_1} \oplus c_2 I_{n_2}$ has $A_{n_1,n_2}$-nonnegative curvature operator of the second kind if and only if $c_1=c_2\geq 0$ by Proposition \ref{prop k1=k2}. 

(2). Apply (1) to $-R$. 

(3). This follows from the fact that $R=c(I_{n_1} \oplus I_{n_2})$ has $\a$-nonnegative or $\a$-nonpositive curvature operator of the second kind for some $\a < A_{n_1,n_2}$ if and only if $c=0$. 
\end{proof}

At last, we give the proof of Theorem \ref{thm product spheres}. 
\begin{proof}[Proof of Theorem \ref{thm product spheres}]
(1). This is an immediate consequence of part (3) of Proposition \ref{prop 5.3}. 

(2). Let $(p_1,p_2)\in M_1\times M_2$. By part (2) of Proposition \ref{prop 5.3}, we have $$R(p_1,p_2)=c(p_1,p_2) I_{n_1}\oplus I_{n_2}.$$
Fixing $p_1$ while letting $p_2$ vary in $M_2$ shows that $c(p_1,p_2)$ is independent of $p_2$. Similarly, $c(p_1,p_2)$ is also independent of $p_1$. 
This shows that both factors have constant sectional curvature $c \geq 0$. 

If $M$ is further assumed to be complete, then $M$ is either flat or isometric to $\S^{n_1}\times \S^{n_2}$, up to scaling.

(3). Similar to the proof of (2). 
\end{proof}

\begin{proof}[Proof of Theorem \ref{thm split}]
Suppose that $(M^n,g)$ splits locally near $q \in M$ as a Riemannian product $(M_1^{k}\times M_2^{n-k}, g_1\oplus g_2)$ with $1\leq k \leq n/2$. 
Then the Riemann curvature tensor $R$ of $M$ satisfies $R=R_1\oplus R_2$ near $q$, where $R_i$ denotes the Riemann curvature tensor of $M_i$ for $i=1,2$.

By part (3) of Proposition \ref{prop 4.3} if $k=1$ and part (3) of Proposition \ref{prop 5.3} if $2\leq k \leq n/2$, the assumption \begin{equation*}
    \a < k(n-k)+\frac{2k(n-k)}{n}
\end{equation*}
implies that $M$ must be flat. 
\end{proof}

\section{Holonomy restriction}

We prove Theorem \ref{thm holonomy} in this section. 

\begin{proof}[Proof of Theorem \ref{thm holonomy}]

Suppose that $(M^n,g)$ splits locally near $q \in M$ as a Riemannian product $(M_1^{k}\times M_2^{n-k}, g_1\oplus g_2)$ with $2\leq k \leq n/2$. 
Then the Riemann curvature tensor $R$ of $M$ satisfies $R=R_1\oplus R_2$ near $q$, where $R_i$ denotes the Riemann curvature tensor of $M_i$ for $i=1,2$.

Noticing that 
$$\a < n+\frac{n-2}{n} \leq  A_{k,n-k}=k(n-k)+\frac{2k(n-k)}{n}$$
for any $1\leq k \leq n/2$, we conclude from part (3) of Propositions \ref{prop 4.3} if $k=1$ and part (3) of Proposition \ref{prop 5.3} if $2\leq k\leq n/2$ that $M$ is locally flat.
Since the restricted holonomy does not depend on $q\in M$, we conclude that $M$ is flat. 
Therefore, $M$ is either locally irreducible or flat. 

If $n=3$, then the holonomy of $M$ must be $\mathsf{SO}(3)$ as $M$ is locally irreducible. So we may assume $n\geq 4$ below.

If $M$ is an irreducible locally symmetric space, then it is Einstein. 
Since $$\a < n+\frac{n-2}{n}\leq \frac{3n}{2}\frac{n+2}{n+4}$$ for any $n \geq 4$, 
we get from \cite[Theorem B]{NPWW22} that either $M$ is flat or the restricted holonomy of $M$ is $\SO(n)$.

So we may assume that $M$ is not locally symmetric with irreducible holonomy representation. Then the restricted holonomy of $M$ is contained in Berger's list of holonomy groups \cite{Berger55}: $\SO(n)$, $\U(\frac{n}{2})$, $\mathsf{SU}(\frac n 2)$, $\mathsf{Sp}(\frac n 4) \mathsf{Sp}(1)$, $\mathsf{Sp}(\frac n 4)$, $\mathsf{G}_2$ and $\mathsf{Spin}_7$.
$M$ must have Ricci flat and thus flat if its restricted holonomy is  $\mathsf{SU}(\frac n 2)$, $\mathsf{Sp}(\frac n 4)$, $\mathsf{G}_2$ or $\mathsf{Spin}_7$.

If the restricted holonomy of $M$ is $\mathsf{Sp}(\frac n 4) \mathsf{Sp}(1)$, then $M$ is quaternion-K\"ahler and thus Einstein. In this case, either the restricted holonomy of $M$ is $\SO(n)$ or $M$ is flat by \cite[Theorem B]{NPWW22}. 

If the restricted holonomy of $M$ is $\U(\frac{n}{2})$, then $M$ is K\"ahler. Noticing that $\a < n+\frac{n-2}{n} \leq \frac{3}{2}\left(\frac{n^2}{4}-1\right)$ for any $n\geq 4$, $M$ must be flat by \cite[Therorem 1.2]{Li22Kahler}. 

Overall, either the restricted holonomy of $M$ is $\SO(n)$ or $M$ is flat. 
\end{proof}

\section{K\"ahler Manifolds}
In this section, we prove Theorem \ref{thm product CPm}. The proof shares the same idea as in the Section 4, but use the orthonormal basis of a complex Euclidean space  constructed in \cite{Li22Kahler} based on $\CP^m$. 

In the following, $B_{m_1,m_2}$ is the expression defined in \eqref{eq B_m_1m_2 def} and $R_{\CP^{m}}$ denotes the Riemann curvature tensor of the complex projective space with constant holomorphic sectional curvature $4$. We establish the following proposition. 
\begin{proposition}\label{prop 6.3}
For $i=1,2$, let $(V_i,g_i,J_i)$ be a complex Euclidean vector space of complex dimension $m_i\geq 1$.  Let $R_i \in S^2_B(\Lambda^2 V_i)$ and $R=R_1\oplus R_2 \in S^2_B(\Lambda^2 (V_1 \times V_2))$. 
\begin{enumerate}
    \item Suppose that $R$ has $B_{m_1,m_2}$-nonnegative curvature operator of the second kind. Then $R=c(R_{\CP^{m_1}} \oplus R_{\CP^{m_2}})$ for some $c\geq 0$.
    \item Suppose that $R$ has $B_{m_1,m_2}$-nonpositive curvature operator of the second kind. Then $R=c(R_{\CP^{m_1}} \oplus R_{\CP^{m_2}})$ for some $c\leq 0$.
    \item Suppose that $R$ has $\a$-nonnegative or $\a$-nonpositive curvature operator of the second kind for some $\a < B_{m_1,m_2}$, then $R$ is flat.
\end{enumerate} 
\end{proposition}

\begin{proof}
(1). Let $\{e_1,\cdots, e_{m_1}, J_1e_1, \cdots, J_1 e_{m_1}\}$ be an orthonormal basis of $(V_1,g_1,J_1)$ and $\{e_{m_1+1},\cdots, e_{m_1+m_2}, J_2e_{m_1+1}, \cdots, J_2 e_{m_1+m_2}\}$ be an orthonormal basis of $(V_2,g_2,J_2)$. 

As in Section 4, we have the orthogonal decomposition 
\begin{equation*}
    S^2_0(V_1\times V_2)=S^2_0(V_1) \oplus S^2_0(V_2) \oplus \spn\{u \odot v: u \in V_1, v\in V_2\} \oplus \R \zeta,
\end{equation*}
where
\begin{equation*}
    \zeta=\frac{1}{\sqrt{2m_1m_2(m_1+m_2)}}(m_2g_1-m_1g_2).
\end{equation*}
The same computation as in Section 4 gives that 
\begin{equation}\label{eq R zeta}
    \mathring{R}(\zeta,\zeta)=-\frac{m_2^2S_1+m_1^2S_2}{2m_1m_2(m_1+m_2)}, 
\end{equation}
where $S_i$ denotes the scalar curvature of $R_i$ for $i=1,2$. 

By Lemma \ref{lemma 3.1}, the subspace $\spn\{u \odot v: u \in V_1, v\in V_2\}$ lies in the kernel of $\mathring{R}$ and its real dimension is $4m_1m_2$. 
%is given by $\{h_{ij}\}_{1\leq i \leq 2m_1, 1\leq j \leq 2m_2}$, where
%\begin{eqnarray*}
%    h_{ij} &=& \frac{1}{\sqrt{2}} e_i \odot e_{m_1+j}, \text{ for } 1\leq i \leq m_1, 1\leq j \leq m_2, \\
%    h_{i,m_2+j}  &=& \frac{1}{\sqrt{2}} e_i \odot J_2 e_{m_1+j}, \text{ for } 1\leq i \leq m_1, 1\leq j \leq m_2, \\
%    h_{m_1+i,j} &=& \frac{1}{\sqrt{2}} J_1e_i \odot e_{m_1+j}, \text{ for } 1\leq i \leq m_1, 1\leq j \leq m_2, \\
%    h_{m_1+i, m_2+j}  &=& \frac{1}{\sqrt{2}} J_1e_i \odot J_2 e_{m_1+j}, \text{ for } 1\leq i \leq m_1, 1\leq j \leq m_2.
%\end{eqnarray*}
%The product structure \eqref{eq R product} implies that 
%\begin{equation}\label{eq R product 4m_1m_2}
%    \sum_{1\leq i \leq 2m_1, 1\leq j\leq 2m_2} \mathring{R}(h_{ij},h_{ij}) =0.
%\end{equation}

For $S^2_0(V_1)$ and $S^2_0(V_2)$, we use the orthonormal bases constructed in Section 4 of \cite{Li22Kahler}. More precisely, the following traceless symmetric two-tensors form an orthonormal basis of $S^2_0(V_1)$: 
\begin{eqnarray*}
    \vp^{1,\pm}_{ij} &=& \frac{1}{2}\left(e_i \odot e_j \mp J_1e_i \odot J_1e_j \right), \text{ for } 1\leq i < j \leq m_1, \\
    \psi^{1,\pm}_{ij} &=& \frac{1}{2}\left(e_i \odot J_1e_j \pm J_1e_i \odot e_j \right), \text{ for } 1\leq i < j \leq m_1, \\
    \a^1_i &=& \frac{1}{2\sqrt{2}}\left(e_i\odot e_i-J_1e_i \odot Je_i \right), \text{ for } 1\leq i \leq m_1, \\
    \a^1_{m_1+i} &=& \frac{1}{\sqrt{2}}\left(e_i\odot J_1e_i \right), \text{ for } 1\leq i \leq m_1, \\
    \eta^1_k &=& \frac{k}{\sqrt{8k(k+1)}}(e_{k+1} \odot e_{k+1} +J_1 e_{k+1} \odot J_1 e_{k+1}) \\
    && -  \frac{1}{\sqrt{8k(k+1)}}\sum_{i=1}^k(e_{i} \odot e_{i} +J_1 e_{i} \odot J_1 e_{i}), \\
    && \text{ for } 1\leq k \leq m_1-1.
\end{eqnarray*}
Similarly, the traceless symmetric two-tensors 
\begin{eqnarray*}
    \vp^{2,\pm}_{ij} &=& \frac{1}{2}\left(e_i \odot e_j \mp J_2e_i \odot J_2e_j \right), \text{ for } m_1+1\leq i < j \leq m_1+m_2, \\
    \psi^{2,\pm}_{ij} &=& \frac{1}{2}\left(e_i \odot J_2e_j \pm J_2e_i \odot e_j \right), \text{ for } m_1+1\leq i < j \leq m_1+m_2, \\
    \a^2_i &=& \frac{1}{2\sqrt{2}}\left(e_i\odot e_i-J_1e_i \odot Je_i \right), \text{ for } m_1+1\leq i \leq m_1+m_2, \\
    \a^2_{m_2+i} &=& \frac{1}{\sqrt{2}}\left(e_i\odot J_1e_i \right), \text{ for } m_1+1\leq i \leq m_1+m_2, \\
    \eta^2_k &=& \frac{k}{\sqrt{8k(k+1)}}(e_{k+1} \odot e_{k+1} +J_2 e_{k+1} \odot J_2 e_{k+1}) \\
    && -  \frac{1}{\sqrt{8k(k+1)}}\sum_{i=1}^k(e_{i} \odot e_{i} +J_2 e_{i} \odot J_2 e_{i}), \\
    && \text{ for } m_1+1\leq k \leq m_1+m_2-1,
\end{eqnarray*}
form an orthonormal basis for $S^2_0(V_2)$. 
Here the superscripts $1$ and $2$ indicate that these are quantities associated to the space $V_1$ and $V_2$, respectively. 

By Lemma 4.3 in \cite{Li22Kahler}, we have 
%\begin{eqnarray}\label{R sum positive eigenvalues 1}
%   &&  \sum_{1\leq i< j \leq m} \left( \mathring{R}(\vp^{1,+}_{ij}, \vp^{1,+}_{ij}) + \mathring{R}(\psi^{1,+}_{ij}, \psi^{1,+}_{ij}) \right) + \sum_{i=1}^{2m_1} \mathring{R}(\a^1_{i}, \a^1_{i}) \\
%  &=& 4 \sum_{1\leq i< j \leq m}  R(e_i, J_1e_i, e_i, J_1e_j) +2 \sum_{i=1}^{m} R(e_i,J_1e_i, e_i, J_1e_i) \nonumber \\
%   &=& S_1. \nonumber
%\end{eqnarray}
\begin{eqnarray}\label{eq sum of negative eigenvaleus 1}
   && \sum_{1\leq i <j \leq m_1} \left( \mathring{R}(\vp^{1,-}_{ij}, \vp^{1,-}_{ij}) + \mathring{R}(\psi^{1,-}_{ij}, \psi^{1,-}_{ij}) \right) +\sum_{k=1}^{m_1-1} \mathring{R}(\eta_k,\eta_k) \\
   % &=& -\frac{m_1-1}{m_1} \sum_{i=1}^{m_1} R(e_i,J_1e_i,e_i,J_1e_i)  -2\frac{m_1-1}{m_1} \sum_{1\leq i <j \leq m_1} R(e_i,J_1e_i,e_j,J_1e_j). \nonumber \\
    &=& -\frac{m_1-1}{2m_1} S_1 \nonumber 
\end{eqnarray}
and
%\begin{eqnarray}\label{R sum positive eigenvalues 2}
%   &&  \sum_{m_1+1\leq i< j \leq m_1+m_2} \left( \mathring{R}(\vp^{2,+}_{ij}, \vp^{2,+}_{ij}) + \mathring{R}(\psi^{2,+}_{ij}, \psi^{2,+}_{ij}) \right) + \sum_{i=m_1+1}^{2m_2+m_1+1} \mathring{R}(\a^2_{i}, \a^2_{i}) \\
%  &=& 4 \sum_{1\leq i< j \leq m}  R(e_i, J_2e_i, e_i, J_2e_j) +2 \sum_{i=m_1+1}^{m_1+m_2} R(e_i,J_2e_i, e_i, J_2e_i) \nonumber \\
%   &=& S_2. \nonumber
%\end{eqnarray}
\begin{eqnarray}\label{eq sum of negative eigenvaleus 2}
   && \sum_{m_1+1\leq i <j \leq m_1+m_2} \left( \mathring{R}(\vp^{2,-}_{ij}, \vp^{2,-}_{ij}) + \mathring{R}(\psi^{2,-}_{ij}, \psi^{2,-}_{ij}) \right) +\sum_{k=m_1+1}^{m_1+m_2-1} \mathring{R}(\eta_k,\eta_k) \\
   % &=& -\frac{m_2-1}{m_2} \sum_{i=m_1+1}^{m_1+m_2} R(e_i,J_2e_i,e_i,J_2e_i)  -2\frac{m_2-1}{m_2} \sum_{m_1+1\leq i <j \leq m_1+m_2-1} R(e_i,J_2e_i,e_j,J_2e_j). \nonumber \\
    &=& -\frac{m_2-1}{2m_2} S_2. \nonumber 
\end{eqnarray}

Combining \eqref{eq R zeta}, \eqref{eq sum of negative eigenvaleus 1} and \eqref{eq sum of negative eigenvaleus 2} together yields 
\begin{eqnarray*}
   && \sum_{1\leq i <j \leq m_1} \left( \mathring{R}(\vp^{1,-}_{ij}, \vp^{1,-}_{ij}) + \mathring{R}(\psi^{1,-}_{ij}, \psi^{1,-}_{ij}) \right)  \\
    && + \sum_{m_1+1\leq i <j \leq m_1+m_2} \left( \mathring{R}(\vp^{2,-}_{ij}, \vp^{2,-}_{ij}) + \mathring{R}(\psi^{2,-}_{ij}, \psi^{2,-}_{ij}) \right) \\
    && +\sum_{k=1}^{m_1-1} \mathring{R}(\eta_k,\eta_k) +\sum_{k=m_1+1}^{m_1+m_2-1} \mathring{R}(\eta_k,\eta_k) +\mathring{R}(\zeta,\zeta) \\
    &=& -\frac{m_1-1}{2m_1}S_1 -\frac{m_2-1}{2m_2}S_2 +\mathring{R}(\zeta,\zeta) \\
    &=& -\frac{1}{2}(m_1^2-1)\bar{a} -\frac{1}{2}(m_2^2-1)\bar{b}-\frac{m_2^2S_1 +m_1^2S_2}{2m_1m_2(m_1+m_2)}\\
    &=& -B_1\bar{a}-B_2\bar{b}, 
\end{eqnarray*}
where we have introduced
\begin{equation*}
    B_1=\frac{1}{2}(m_1^2-1)+ \frac{(m_1+1)m_2}{2(m_1+m_2)} \text{ and } B_2=\frac{1}{2}(m_2^2-1)+\frac{(m_2+1)m_1}{2(m_1+m_2)}
\end{equation*}
for simplicity of notations. Note that $-B_1\bar{a}-B_2\bar{b}$ is the sum of $$1+4m_1m_2+(m_1^2-1)+(m_2^2-1)$$-many diagonal elements of the matrix representation of $\mathring{R}$ with respect to the orthonormal basis of $S^2_0(V_1\times V_2)$ constructed above (here one can pick any orthonormal basis for the subspace $\spn\{u \odot v: u \in V_1, v\in V_2 \}$ as it is in the kernel of $\mathring{R}$).

Let $A$ be the collection of the values $\mathring{R}(\a^1_i,\a^1_i)$ for $1\leq i \leq 2m_1$, $\mathring{R}(\vp^{1,+}_{ij},\vp^{1,+}_{ij})$ and $\mathring{R}(\psi^{1,+}_{ij},\psi^{1,+}_{ij})$ for $1\leq i < j \leq m$.
By Lemma 4.3 in \cite{Li22Kahler}, we know that $A$ contains two copies of $R(e_i,J_1e_i,e_i,J_1e_i)$ for each $1\leq i \leq m_1$ and two copies of $2R(e_i,J_1e_i,e_j,J_2e_j)$ for each $1\leq i < j \leq m_1$. 
Therefore, the sum of all values in $A$ is equal to $S_1$, the scalar curvature of $R_1$, and  $\bar{a}$, the average of all values in $A$, is given by 
\begin{equation*}
    \bar{a}=\frac{S_1}{m_1(m_1+1)}.
\end{equation*}

Similarly, let $B$ be the collection of the values $\mathring{R}(\a^2_i,\a^2_i)$ for $m_1+1\leq i \leq m_1+2m_2$, $\mathring{R}(\vp^{2,+}_{ij},\vp^{2,+}_{ij})$ and $\mathring{R}(\psi^{2,+}_{ij},\psi^{2,+}_{ij})$ for $m_1+1\leq i < j \leq m_1+m_2$.
By Lemma 4.3 in \cite{Li22Kahler}, we know that $B$ contains two copies of $R(e_i,J_2e_i,e_i,J_2e_i)$ for each $m_1+1\leq i \leq m_1+m_2$ and two copies of $2R(e_i,J_2e_i,e_j,J_2e_j)$ for each $m_1+1\leq i < j \leq m_1+m_2$. 
Therefore, the sum of all values in $B$ is equal to $S_2$, the scalar curvature of $R_2$, and  $\bar{b}$, the average of all values in $B$, is given by 
\begin{equation*}
    \bar{b}=\frac{S_2}{m_2(m_2+1)}.
\end{equation*}

%By the calculation in Section 4, 
%\begin{equation}
%    \mathring{R}(\zeta,\zeta)= -\frac{m_2^2S_1 +m_1^2S_2}{2m_1m_2(m_1+m_2)}
%\end{equation}

Noticing that 
$$B_{m_1,m_2}=1+(m_1^2-1)+(m_2^2-1)+4m_1m_2+B_1+B_2,$$ 
the assumption $R$ has $B_{m_1,m_2}$-nonnegative curvature operator of the second kind implies that
\begin{eqnarray}\label{eq key 3}
 B_1\bar{a}+ B_2\bar{b}  &\leq & f(A, \lfloor B_1\rfloor)+f(B, B_1+B_2 -\lfloor B_1\rfloor) \\ \nonumber
 &\leq&  \lfloor B_1\rfloor \bar{a} + (B_1+B_2 -\lfloor B_1\rfloor) \bar{b} \\ \nonumber
 &=& B_1\bar{a}+ B_2\bar{b}  + (B_1-\lfloor B_1\rfloor)(\bar{b}-\bar{a})
\end{eqnarray}
and 
\begin{eqnarray}\label{eq key 4}
 B_1\bar{a}+ B_2\bar{b}  &\leq & f(A, B_1+B_2 -\lfloor B_2\rfloor)+f(B, \lfloor B_2\rfloor) \\ \nonumber
 &\leq&  (B_1+B_2 -\lfloor B_2\rfloor) \bar{a} + \lfloor B_2\rfloor \bar{b} \\ \nonumber
 &=& B_1\bar{a}+ B_2\bar{b}  + (B_2-\lfloor B_2\rfloor)(\bar{a}-\bar{b}),
\end{eqnarray}
where $f$ is the function defined in Lemma \ref{lemma average} and we have used Lemma \ref{lemma average} to estimate $f$.  
So we get from \eqref{eq key 3} if $\bar{a}\geq \bar{b}$ and from \eqref{eq key 4} if  $\bar{a}\leq \bar{b}$ that 
\begin{equation*}
    B_1 \bar{a} +B_2 \bar{b}  \leq B_1 \bar{a} +B_2 \bar{b}.
\end{equation*}
Therefore, either in \eqref{eq key 3} or \eqref{eq key 4}, we must have equalities in the inequalities used for $f$. By Lemma \ref{lemma average}, we get that all the values in $A$ are equal to $\bar{a}$ and all the values in $B$ are equal to $\bar{b}$.  
Hence, both $R_1$ and $R_2$ have constant holomorphic sectional curvature, that is to say, $R=c_1 R_{\CP^{m_1}} \oplus c_2 R_{\CP^{m_2}}$ for $c_1, c_2 \in \R$. 

Finally, we must have $c_1=c_2\geq 0$, as $R=c_1 R_{\CP^{m_1}} \oplus c_2 R_{\CP^{m_2}}$ has $B_{m_1,m_2}$-nonnegative curvature operator of the second kind if and only if $c_1=c_2\geq 0$ by Proposition \ref{prop k1=k2}.

(2). Apply (1) to $-R$. 

(3). This follows from the fact that $R=c(R_{\CP^{m_1}} \oplus R_{\CP^{m_2}})$ has $\a$-nonnegative or $\a$-nonpositive curvature operator of the second kind for some $\a < B_{m_1,m_2}$ if and only if $c=0$. 
\end{proof}

\begin{proof}[Proof of Theorem \ref{thm product CPm}]
Once we have Proposition \ref{prop 6.3}, this is similar to the proof of Theorem \ref{thm product spheres}. 

%(1). This is an immediate consequence of part (3) of Proposition \ref{prop 6.3}. 

%(2). Let $(p_1,p_2)\in M_1\times M_2$. By part (2) of Proposition \ref{prop 6.3}, we have $$R(p_1,p_2)=c(p_1,p_2) (R_{\CP^{m_1}} \oplus R_{\CP^{m_2}}).$$
%Fixing $p_1$ and letting $p_2$ vary in $M_2$ shows that $c(p_1,p_2)$ is independent of $p_2$. Similarly, $c(p_1,p_2)$ is independent of $p_1$ as well. 
%This shows that both factors have constant holomorphic sectional curvature $c \geq 0$. 
%Hence, $M$ is either flat or isometric to $\CP^{m_1}\times \CP^{m_2}$, up to scaling.

%(3). Similar to the proof of (2). 
\end{proof}
%\input{7Kahler2D.tex}

%\section*{Acknowledgments}
%The author would like to thank Professor Xiaodong Cao and Professor Hung Tran for some helpful discussions. 

\bibliographystyle{alpha}
\bibliography{ref}

\begin{thebibliography}{NPWW22}

\bibitem[Ber55]{Berger55}
Marcel Berger.
\newblock Sur les groupes d'holonomie homog\`ene des vari\'{e}t\'{e}s \`a
  connexion affine et des vari\'{e}t\'{e}s riemanniennes.
\newblock {\em Bull. Soc. Math. France}, 83:279--330, 1955.

\bibitem[BK78]{BK78}
Jean-Pierre Bourguignon and Hermann Karcher.
\newblock Curvature operators: pinching estimates and geometric examples.
\newblock {\em Ann. Sci. \'{E}cole Norm. Sup. (4)}, 11(1):71--92, 1978.

\bibitem[Bre08]{Brendle08}
Simon Brendle.
\newblock A general convergence result for the {R}icci flow in higher
  dimensions.
\newblock {\em Duke Math. J.}, 145(3):585--601, 2008.

\bibitem[CGT21]{CGT21}
Xiaodong Cao, Matthew~J. Gursky, and Hung Tran.
\newblock Curvature of the second kind and a conjecture of {N}ishikawa.
\newblock {\em arXiv:2112.01212}, 2021.

\bibitem[Kas93]{Kashiwada93}
Toyoko Kashiwada.
\newblock On the curvature operator of the second kind.
\newblock {\em Natur. Sci. Rep. Ochanomizu Univ.}, 44(2):69--73, 1993.

\bibitem[Li21]{Li21}
Xiaolong Li.
\newblock Manifolds with nonnegative curvature operator of the second kind.
\newblock {\em arXiv:2112.08465v4}, 2021.

\bibitem[Li22a]{Li22Kahler}
Xiaolong Li.
\newblock K\"ahler manifolds and the curvature operator of the second kind.
\newblock {\em arXiv:2208.14505}, 2022.

\bibitem[Li22b]{Li22PAMS}
Xiaolong Li.
\newblock K\"ahler surfaces with six-positive curvature operator of the second
  kind.
\newblock {\em arXiv:2207.00520}, 2022.

\bibitem[Li22c]{Li22JGA}
Xiaolong Li.
\newblock Manifolds with $4\frac{1}{2}$-positive curvature operator of the
  second kind.
\newblock {\em J. Geom. Anal., to appear in the special volume "Analysis and
  Geometry of Complete Manifolds" in honor of Professor Peter Li's 70th
  birthday, arXiv:2206.15011}, 2022.

\bibitem[MM88]{MM88}
Mario~J. Micallef and John~Douglas Moore.
\newblock Minimal two-spheres and the topology of manifolds with positive
  curvature on totally isotropic two-planes.
\newblock {\em Ann. of Math. (2)}, 127(1):199--227, 1988.

\bibitem[Nis86]{Nishikawa86}
Seiki Nishikawa.
\newblock On deformation of {R}iemannian metrics and manifolds with positive
  curvature operator.
\newblock In {\em Curvature and topology of {R}iemannian manifolds ({K}atata,
  1985)}, volume 1201 of {\em Lecture Notes in Math.}, pages 202--211.
  Springer, Berlin, 1986.

\bibitem[NPW22]{NPW22}
Jan Nienhaus, Peter Petersen, and Matthias Wink.
\newblock Betti numbers and the curvature operator of the second kind.
\newblock {\em arXiv:2206.14218}, 2022.

\bibitem[NPWW22]{NPWW22}
Jan Nienhaus, Peter Petersen, Matthias Wink, and William Wylie.
\newblock Holonomy restrictions from the curvature operator of the second kind.
\newblock {\em arXiv:2208.13820}, 2022.

\bibitem[OT79]{OT79}
Koichi Ogiue and Shun-ichi Tachibana.
\newblock Les vari\'{e}t\'{e}s riemanniennes dont l'op\'{e}rateur de courbure
  restreint est positif sont des sph\`eres d'homologie r\'{e}elle.
\newblock {\em C. R. Acad. Sci. Paris S\'{e}r. A-B}, 289(1):A29--A30, 1979.

\end{thebibliography}

\end{document}